\documentclass[english,12pt,a4paper,reqno]{amsart} 

\usepackage{fullpage}
\usepackage{babel}
\usepackage{amssymb}
\usepackage{amsmath}
\usepackage{stmaryrd}
\usepackage{pifont}
\usepackage{textcomp}
\usepackage[mathscr]{euscript}
\usepackage{hyperref}
\usepackage[small,nohug]{diagrams}
\diagramstyle[labelstyle=\scriptstyle]

\theoremstyle{plain}
\newtheorem{Theorem}{Theorem}[section]
\newtheorem{Proposition}[Theorem]{Proposition}
\newtheorem{Lemma}[Theorem]{Lemma}

\newtheorem{Corollary}[Theorem]{Corollary}
\newtheorem{Example}[Theorem]{Example}

\newtheorem{Setting}[Theorem]{Setting}

\theoremstyle{remark}
\newtheorem{Definition}[Theorem]{Definition}
\newtheorem{Remark}[Theorem]{Remark}

\newcommand{\mfS}{\mathfrak{S}}

\newcommand{\liminv}{\varprojlim}
\newcommand{\normal}{\triangleleft}
\newcommand{\Gal}{{\rm Gal}}
\newcommand{\GalS}{{{\bf Gal}_\mfS}}

\newcommand{\Q}{{\mathbb{Q}}}

\newcommand{\rk}{{\rm rk}}

\newcommand{\function}{{:}\;}

\newcommand{\PSCC}{{{\rm P}\mfS{\rm CC}}}

\newcommand{\KtotS}{{K^\mfS}}

\newcommand{\mf}{\mathfrak}

\newcommand{\mfp}{{\mf p}}
\newcommand{\mfq}{{\mf q}}
\newcommand{\mfP}{{\mf P}}
\newcommand{\mfQ}{{\mf Q}}
\newcommand{\mcS}{{\mathcal{S}}}

\newcommand{\mcG}{{\mathcal{G}}}
\newcommand{\mcH}{{\mathcal{H}}}
\newcommand{\mcA}{{\mathcal{A}}}
\newcommand{\mcB}{{\mathcal{B}}}
\newcommand{\mcO}{{\mathcal{O}}}
\newcommand{\mcL}{{\mathcal{L}}}
\newcommand{\mcLr}{{\mcL_{\rm ring}}}

\newcommand{\mbG}{{\mathbf{G}}}
\newcommand{\mbH}{{\mathbf{H}}}
\newcommand{\mbA}{{\mathbf{A}}}
\newcommand{\mbB}{{\mathbf{B}}}

\newcommand{\CC}{{\rm CC}}
\newcommand{\mfpinS}{{\mfp\in\mfS}}

\def\freeprod{
\def\freeprodD{{\prod\kern-13pt{*}\kern5pt}}
\def\freeprodT{\prod\kern-11pt{*}\kern2pt}
\mathop{\mathchoice{\freeprodD}{\freeprodT}
{not defined}{not defined}}
}
\def\hefresh{
   \def\hefreshD{\mathop{\raise1.5pt\hbox{${\smallsetminus}$}}}
   \def\hefreshS{\mathop{\raise0.85pt\hbox{$\scriptstyle\smallsetminus$}}}
   \mathchoice{\hefreshD}{\hefreshD}{\hefreshS}{\hefreshS}}
\renewcommand{\setminus}{\hefresh}

\def\dotunion{
\def\dotunionD{\bigcup\kern-10.5pt\cdot\kern5pt}
\def\dotunionT{\bigcup\kern-8.5pt\cdot\kern3.5pt}
\mathop{\mathchoice{\dotunionD}{\dotunionT}{}{}}}
\renewcommand{\biguplus}{\dotunion}

\newcommand{\bfsigma}{\mbox{\boldmath$\sigma$}}

\newcommand{\smallbfsigma}{\mbox{\boldmath$\scriptstyle\sigma$}}

\newcommand{\rmdef}{\circ}
\newcommand{\equivleq}{\sim}

\begin{document}

\author{Arno Fehm}
\title[Elementary theory of totally $\mfS$-adic numbers]{The elementary theory of\\large fields of totally $\mfS$-adic numbers}

\begin{abstract}
We analyze the elementary theory of certain fields $\KtotS(\bfsigma)$ 
of totally $\mfS$-adic algebraic numbers
that were introduced and studied
by Geyer-Jarden and Haran-Jarden-Pop.
In particular, we provide an axiomatization of these theories and prove their decidability,
thereby giving a common generalization of classical decidability results of Jarden-Kiehne, Fried-Haran-V\"olklein and Ershov.
\end{abstract}

\maketitle

\bibliographystyle{alpha}

\section{Introduction}

\noindent
Let $\mfS$ be a finite set of absolute values on a number field $K$.
By $\KtotS$ we denote the field of totally $\mfS$-adic numbers -- the maximal Galois extension of $K$ in which the elements of $\mfS$ are totally split.
For an integer $e\geq0$ and an $e$-tuple $\bfsigma=(\sigma_1,\dots,\sigma_e)\in\Gal(K)^e$ of elements of the absolute Galois group of $K$,
we let $\KtotS(\bfsigma)$ be the fixed field of the group $\left<\sigma_1,\dots,\sigma_e\right>\leq\Gal(K)$ inside $\KtotS$.

These fields $\KtotS(\bfsigma)$ were studied by Jarden-Razon \cite{JardenRazon}, Geyer-Jarden \cite{GeyerJarden},
and recently in a series of papers by Haran-Jarden-Pop \cite{HJPd}, \cite{HJPe}.
In particular, these authors prove that for almost all $\bfsigma\in\Gal(K)^e$, in the sense of Haar measure on the compact group $\Gal(K)^e$, 
the field $\KtotS(\bfsigma)$ satisfies a local-global principle for rational points on varieties, and its absolute Galois group has a nice description as a free product of local factors.

Combining these results we are able to give an axiomatization 
of the theory $T_{{\rm almost},\mfS,e}$ of first-order sentences in the language of rings that hold in almost all $\KtotS(\bfsigma)$ (Theorem~\ref{almostaxioms})
and we prove the decidability of this theory (Theorem \ref{TheoremSigma}):

\begin{Theorem}\label{TheoremIntro1}
Let $\mfS$ be a finite set of absolute values on a number field $K$, and let $e\geq0$.
Then the first-order theory $T_{{\rm almost},\mfS,e}$ of almost all $\KtotS(\bfsigma)$, $\bfsigma\in\Gal(K)^e$, is decidable.
\end{Theorem}

This theorem is a common generalization of classical decidability results of
Jarden-Kiehne \cite{JardenKiehne} (the case $\mfS=\emptyset$),
Fried-Haran-V\"olklein \cite{FHV94} (the case $K=\mathbb{Q}$, $e=0$ and $\mfS$ consisting only of the archimedean absolute value),
and Ershov \cite{ErshovNLGF} (the case $K=\mathbb{Q}$, $e=0$ and $\mfS$ consisting only of $p$-adic absolute values).

In fact, we prove a more general and stronger statement, see Theorem \ref{TheoremSigma}.
The main part of the proof consists of an analysis of the absolute Galois group of $\KtotS(\bfsigma)$ together with local data,
and the model theory of such structures.

\section{Preliminaries on profinite groups and spaces}
\label{secprelim}
\label{secprofinite}

\noindent
We assume that the reader is familiar with the basic theory of profinite groups, 
as presented in \cite[Ch.~1]{FJ3}, \cite[Ch.~2]{RibesZalesskii}.

We always consider profinite groups as topological groups,
so in particular homomorphisms between profinite groups are continuous group homomorphisms.
By $H\leq G$ (resp.~$H\normal G$) we indicate that $H$ is a closed (resp.~normal closed) subgroup of $G$.
If $X\subseteq G$, we denote by $\left< X \right>$ the closed subgroup
generated by $X$ in $G$.
We use the symbol $1$ to denote both the unit element of $G$,
and the trivial subgroup $\{1\}\leq G$.

In the category of profinite groups,  
direct products,
inverse limits 
and fibre products exist, \cite[22.2.1]{FJ3}.
For the notion of the {\bf rank} $\rk(G)$ of a profinite group $G$
and the notion of a {\bf free} profinite group see \cite[Ch.~17]{FJ3}.
We denote by $\hat{F}_e$ the free profinite group of rank $e$.

\begin{Lemma}\label{GaschuetzSubgroup}
Let $\pi\function G\rightarrow H$ be an epimorphism of profinite groups.
Let $e\geq0$,
let $N\normal G$ be a closed normal subgroup with $\rk(G/N)\leq e$,
and let $h_1,\dots,h_e\in H$ such that $H=\left<h_1,\dots,h_e,\pi(N)\right>$.
Then there exist $g_1,\dots,g_e\in G$ such that
$G=\left<g_1,\dots,g_e,N\right>$ and
$\pi(g_i)=h_i$, $i=1,\dots,e$.
\end{Lemma}

\begin{proof}
Let $\bar{G}=G/N$, $\bar{H}=H/\pi(N)$, and let $\bar{\pi}\function \bar{G}\rightarrow\bar{H}$
be the induced epimorphism.
Then $\bar{H}=\left<\bar{h}_1,\dots,\bar{h}_e\right>$, so Gasch\"utz' lemma \cite[17.7.2]{FJ3}
implies that there are $g_1,\dots,g_e\in G$ such that
$\bar{G}=\left<\bar{g}_1,\dots,\bar{g}_e\right>$ and
$\bar{\pi}(\bar{g}_i)=\bar{h}_i$, $i=1,\dots,e$.
So, $G=\left<g_1,\dots,g_e,N\right>$ and there are $n_1,\dots,n_e\in N$ such that
$\pi(g_i)=h_i\pi(n_i)$, $i=1,\dots,e$.
Thus, setting $g_i^\prime=g_in_i^{-1}$, we have
$G=\left<g_1^\prime,\dots,g_e^\prime,N\right>$
and $\pi(g_i^\prime)=h_i$, $i=1,\dots,e$.
\end{proof}

A {\bf profinite space} is a totally disconnected compact Hausdorff space.
Profinite spaces can be characterized as
inverse limits of finite discrete spaces, 
or as zero-dimensional compact Hausdorff spaces, \cite[1.1.12]{RibesZalesskii}.
Any product and any finite coproduct (i.e.~direct sum) of profinite spaces is a profinite space,
and a subspace of a profinite space is profinite if and only if it is closed.
Since profinite spaces are compact Hausdorff,
any continuous map between profinite spaces is closed,
and any continuous bijection of profinite spaces is a homeomorphism.

\section{Group Piles}
\label{secpiles}

\noindent
The notion of group piles was introduced in \cite{HJPd} to enrich profinite groups with extra local data.
We recall this notion and extend it.
Our main innovation is the introduction of a certain quotient $\bar{\mathbf{G}}$ 
that measures the failure of a deficient group pile $\mathbf{G}$ to be self-generated.

Fix a finite set $\mfS$ not containing $0$ and let $e\geq0$.

\begin{Definition}\label{Defsubgrtop}
Let $G=\liminv _N G/N$ be a profinite group,
where $N$ runs over all open normal subgroups of $G$.
Then the set ${\rm Subgr}(G)$ of all closed subgroups of $G$
is equipped with a profinite topology, 
induced by ${\rm Subgr}(G)=\liminv_N{\rm Subgr}(G/N)$.
The group $G$ acts continuously on ${\rm Subgr}(G)$ by conjugation.
A homomorphism $\alpha\function G\rightarrow H$ of profinite groups
induces a map
${\rm Subgr}(\alpha)\function {\rm Subgr}(G)\rightarrow{\rm Subgr}(H)$
given by $\Gamma\mapsto\alpha(\Gamma)$.
\end{Definition}

\begin{Lemma}\label{Subgrfunctor}
The map ${\rm Subgr}$ is a covariant functor from the category 
of profinite groups (with homomorphisms) 
to the category of profinite spaces
(with continuous maps).
\end{Lemma}

\begin{proof}
It is easy to check that if
$\alpha\function G\rightarrow H$ is a homomorphism of profinite groups,
then the induced map ${\rm Subgr}(G)\rightarrow{\rm Subgr}(H)$ is continuous.
\end{proof}

\begin{Lemma}\label{continuous}
If $H$ is a closed subgroup of a profinite group $G$,
then ${\rm Subgr}(H)$ is a closed subspace of ${\rm Subgr}(G)$.
\end{Lemma}

\begin{proof}
By Lemma~\ref{Subgrfunctor},
the inclusion $\iota\colon{\rm Subgr}(H)\rightarrow{\rm Subgr}(G)$ is continuous.
Since both spaces are compact Hausdorff, $\iota$ is closed,
and thus a topological embedding.
\end{proof}

\begin{Definition}\label{Defreduced}\label{Deforder}
A {\bf group pile} is a structure 
$\mathbf{G}=(G,\mcG_0,\mathcal{G}_\mfp)_{\mfpinS}$
consisting of 
\begin{enumerate}
 \item a profinite group $G$,
 \item a nonempty $G$-invariant closed subset $\mcG_0\subseteq{\rm Subgr}(G)$ such that
        the elements of $\mcG_0$ are pairwise conjugate in $G$, and
 \item a $G$-invariant closed subset $\mathcal{G}_\mfp\subseteq{\rm Subgr}(G)$ for each $\mfpinS$.
\end{enumerate}
The {\bf order} resp.~{\bf rank} of $\mathbf{G}$
is the order resp.~rank of $G$.
A {\bf finite} group pile is a group pile of finite order.
Let $\mathcal{G}=\bigcup_{\mfpinS}\mathcal{G}_\mfp$.
We call $\mathbf{G}$
{\bf self-generated} if
there exists $G_0\in\mcG_0$ such that
$G=\left<G_0,\mathcal{G}\right>$, i.e.~$G$ is generated by $G_0$ and the groups in $\mathcal{G}_\mfp$, $\mfpinS$.
It is called {\bf bare} if $\mathcal{G}=\{1\}$, and {\bf deficient} if $\mcG_0=\{1\}$.
The {\bf deficient reduct} of $\mbG$ is 
$\mbG^{\rmdef}=(G,\{1\},\mcG_\mfp)_\mfpinS$.
Instead of $(G,\{1\},\mcG_\mfp)_\mfpinS$, we also write $(G,\mcG_\mfp)_\mfpinS$.
We call $\mbG$ {\bf separated} if the sets $\mcG_\mfp$, $\mfp\in \{0\}\cup\mfS$, are disjoint,
and {\bf reduced} if there are no non-trivial inclusions among the elements of $\mcG$.
\end{Definition}

\begin{Remark}
Note that if $\mbG$ is self-generated, then
$G=\left<G_0,\mcG\right>$ for {\em any} $G_0\in\mcG_0$.
Condition (2) says that $\mcG_0$ consists of a single $G$-orbit in ${\rm Subgr}(G)$,
i.e.~there exists $G_0\in\mcG_0$ such that $\mcG_0=(G_0)^G:=\{ (G_0)^g \colon g\in G\}$.
Hence, our notion of group piles coincides with the group piles of \cite{HJPd},
except for a small difference in notation concerning $\mcG_0$.
The notion of group piles is
also related to the 
\textquoteleft$\Delta^*$-groups\textquoteright~in 
\cite{ErshovCLEP2}, \cite{ErshovFree}, and \cite{ErshovUniformly}.
\end{Remark}

\begin{Definition}
A {\bf homomorphism} of group piles
$$
 f\function(G,\mcG_0,\mathcal{G}_\mfp)_{\mfpinS}\rightarrow(H,\mcH_0,\mathcal{H}_\mfp)_{\mfpinS}
$$ 
is a homomorphism of profinite groups $f\function G\rightarrow H$
such that 
$f(\mathcal{G}_\mfp)\subseteq\mathcal{H}_\mfp$ for each $\mfp\in \{0\}\cup\mfS$.
It is an {\bf epimorphism} if $f\function G\rightarrow H$ is surjective and $f(\mathcal{G}_\mfp)=\mathcal{H}_\mfp$ for each $\mfp\in \{0\}\cup\mfS$.
It is an {\bf isomorphism} if in addition $f\function G\rightarrow H$
is an isomorphism.
The homomorphism $f$ is called {\bf rigid}
if $f|_\Gamma$ is injective for each $\Gamma\in\mathcal{G}$.
If $N$ is a closed normal subgroup of $G$, define
the {\bf quotient} 
$\mathbf{G}/N=(G/N,\mcG_{0,N},\mathcal{G}_{\mfp,N})_{\mfpinS}$ 
by 
$\mathcal{G}_{\mfp,N}=\left\{ \Gamma N/N \colon \Gamma\in\mathcal{G}_\mfp \right\} \subseteq {\rm Subgr}(G/N)$.
This is again a group pile,
the quotient map $G\rightarrow G/N$ extends to an epimorphism of group piles
$\mbG\rightarrow\mbG/N$,
and every epimorphism of group piles is of this form.
\end{Definition}

\begin{Remark}
We identify the category of bare deficient group piles (with homomorphisms)
with the category
of profinite groups (with homomorphisms)
via the forgetful functor $(G,\mcG_0,\mathcal{G}_\mfp)_{\mfpinS}\mapsto G$.
\end{Remark}

\begin{Lemma}\label{lem:invlimgp}
In the category of group piles with epimorphisms, inverse limits exist.
\end{Lemma}

\begin{proof}
For a directed set $I$ and
an inverse family $\mathbf{G}_i=(G_i,\mcG_{i,0},\mcG_{i,\mfp})_\mfpinS$, $i\in I$, of group piles,
$\mathbf{G}=(G,\mcG_0,\mcG_\mfp)_\mfpinS$ with
$G:=\varprojlim_{i\in I} G_i$ and $\mathcal{G}_\mfp:=\varprojlim_{i\in I}\mathcal{G}_{i,\mfp}\subseteq{\rm Subgr}(G)$, $\mfpinS\cup\{0\}$,
is an inverse limit.
\end{proof}

\begin{Definition}\label{defebounded}
For a group pile $\mathbf{G}=(G,\mcG_0,\mathcal{G}_\mfp)_{\mfpinS}$
let $G^\prime:=\left<\mathcal{G}\right>$
be the closed subgroup generated by the subgroups in $\mathcal{G}_\mfp$, $\mfpinS$.
Let $\mathbf{G}^\prime:=(G^\prime,\mathcal{G}_\mfp)_{\mfpinS}$ 
and
$\bar{\mathbf{G}}:=\mathbf{G}/G^\prime$.
We say that $\mbG$ is
{\bf $e$-generated}
if $\rk(\bar\mbG)\leq e$,
and
{\bf $e$-bounded} if $\mbG$ is self-generated and $\rk(G_0)\leq e$ for all $G_0\in\mcG_0$.
\end{Definition}

\begin{Lemma}
$\mathbf{G}^\prime$ is a self-generated and deficient group pile,
and $\bar{\mathbf{G}}$ is a bare group pile.
\end{Lemma}

\begin{proof}
By Lemma~\ref{continuous}, $\mathbf{G}^\prime$ is a group pile. The other claims are obvious.
\end{proof}

\begin{Lemma}\label{primesubgroups}
If $\varphi\function\mathbf{G}\rightarrow\mathbf{H}$
is an epimorphism of group piles,
then $\varphi(G^\prime)=H^\prime$,
so $\varphi$ induces epimorphisms
$\varphi^\prime\function\mathbf{G}^\prime\rightarrow\mathbf{H}^\prime$
and
$\bar{\varphi}\function\bar{\mathbf{G}}\rightarrow\bar{\mathbf{H}}$.
\end{Lemma}

\begin{proof}
By definition of an epimorphism of group piles,
$\varphi(\mathcal{G})=\mathcal{H}$.
Hence, since $\varphi$ is continuous and closed,
$\varphi(\left<\mcG\right>)=\left<\mcH\right>$,
as claimed.
\end{proof}

\begin{Lemma}\label{functoriality}
The map $\mathbf{G}\mapsto\mathbf{G}^\prime$ (resp.~$\mathbf{G}\mapsto\bar{\mathbf{G}}$) is a covariant functor from the
category of group piles with epimorphisms 
to the category of self-generated deficient group piles with epimorphisms
(resp.~the category of bare group piles with epimorphisms).
\end{Lemma}

\begin{proof}
This follows from Lemma~\ref{primesubgroups}.
\end{proof}

\begin{Lemma}\label{barecorrespondence}
Let $\mathbf{G}=(G,\mathcal{G}_\mfp)_{\mfpinS}$ be a deficient group pile and 
$\mathbf{A}=A$ a bare deficient group pile.
Then the map $\varphi\mapsto\bar{\varphi}$
gives a bijection between the
epimorphisms from $\mathbf{G}$ to $\mathbf{A}$ and the
epimorphisms from $\bar{G}$ to $A$.
\end{Lemma}

\begin{proof}
If $\varphi\function\mathbf{G}\rightarrow\mathbf{A}$ is an epimorphism, 
then $\bar{\varphi}\function\bar{\mathbf{G}}\rightarrow\bar{\mathbf{A}}=\mathbf{A}$
is an epimorphism.
Conversely, given an epimorphism $\bar{\varphi}\function\bar{G}\rightarrow A$,
the composition $\bar{\varphi}\circ\pi\colon\mathbf{G}\rightarrow\mathbf{A}$, where $\pi\function\mathbf{G}\rightarrow\bar{\mathbf{G}}$ is the quotient map,
is an epimorphism. 
These two operations are inverse to each other.
\end{proof}

\begin{Remark}
Note that a deficient group pile is self-generated
if and only if it is $0$-generated.
Every $e$-bounded group pile is $e$-generated.
\end{Remark}

\begin{Lemma}\label{quotientrank}
Let $\varphi\function\mathbf{G}\rightarrow\mathbf{H}$
be an epimorphism of group piles.
If $\mathbf{G}$ is $e$-generated, 
then $\mathbf{H}$ is $e$-generated.
If $\mbG$ is $e$-bounded, then $\mbH$ is $e$-bounded.
\end{Lemma}

\begin{proof}
The induced map $\bar{\varphi}\function\bar{\mathbf{G}}\rightarrow\bar{\mathbf{H}}$
is an epimorphism by Lemma~\ref{functoriality}, so $\rk(\bar{H})\leq\rk(\bar{G})$.
If $\mbG$ is self-generated, then also $\mbH$ is self-generated.
Since $\varphi(\mcG_0)=\mcH_0$, 
for $H_0\in\mcH_0$ there exists $G_0\in\mcG_0$ with $\varphi(G_0)=H_0$
and thus $\rk(H_0)\leq\rk(G_0)$.
\end{proof}

\begin{Proposition}\label{egeneratedcharacterization}
A deficient group pile $\mathbf{G}$ is $e$-generated if and only if every finite quotient of $\mathbf{G}$
is $e$-generated.
\end{Proposition}

\begin{proof}
If $\mathbf{G}$ is $e$-generated, then every finite quotient of $\mathbf{G}$ is $e$-generated by Lemma~\ref{quotientrank}.
Conversely, suppose that $\mathbf{G}$ is not $e$-generated.
Then there is an epimorphism $\bar{G}\rightarrow A$
onto a finite group $A$ with $\rk(A)>e$ (see for example \cite[2.5.3]{RibesZalesskii}),
so $A$ is a finite quotient of $\mathbf{G}$ (Lemma~\ref{barecorrespondence}) which is not $e$-generated.
\end{proof}

\begin{Lemma}\label{GaschuetzPiles}
Let $\mbA$ be an $e$-bounded self-generated group pile and
let $\tilde{\mbB}=(B,\mathcal{B}_\mfp)_\mfpinS$ be an $e$-generated deficient group pile.
For every epimorphism $\pi\function\tilde{\mbB}\rightarrow\mathbf{A}^\rmdef$ 
there exists an $e$-bounded self-generated group pile $\mbB$ with $\mbB^{\rmdef}=\tilde{\mbB}$
such that $\pi\function\mbB\rightarrow\mbA$ is an epimorphism.
\end{Lemma}

\begin{proof}
Let $A_0\in\mcA_0$ and choose $a_1,\dots,a_e\in A$ with $A_0=\left<a_1,\dots,a_e\right>$.
By Lemma~\ref{primesubgroups}, $A=\left<a_1,\dots,a_e,A^\prime\right>$ and $A^\prime=\pi(B^\prime)$,
so Lemma~\ref{GaschuetzSubgroup} gives
$b_1,\dots,b_e\in B$ with $B=\left<b_1,\dots,b_e,B^\prime\right>$ and $\pi(b_i)=a_i$.
Let $B_0=\left<b_1,\dots,b_e\right>$ and 
$\mbB=(B,(B_0)^B,\mcB_\mfp)_\mfpinS$.
Then $\mbB$ is $e$-bounded and 
$\pi\function\mbB\rightarrow\mbA$ is an epimorphism.
\end{proof}

\section{Embedding Problems}

\noindent
We recall the notion of embedding problems for group piles from \cite[\S4]{HJPd}
and rephrase some results in terms of $e$-bounded group piles.

\begin{Definition}\label{DefEP}
Let $\mbG$ be a group pile.
An {\bf embedding problem} for $\mathbf{G}$ is a pair
$$
 EP=(\varphi\function\mathbf{G}\rightarrow\mathbf{A},\;\alpha\function\mathbf{B}\rightarrow\mathbf{A})
$$
of epimorphisms of group piles.
It is called {\bf finite}, {\bf self-generated},
{\bf $e$-generated}, {\bf $e$-bounded}, {\bf deficient}, or {\bf bare},
if $\mathbf{B}$ 
has this property.
It is called {\bf rigid} 
if $\alpha$ is rigid.
A {\bf solution} of the embedding problem $(\varphi,\alpha)$
is an epimorphism $\gamma\function\mathbf{G}\rightarrow\mathbf{B}$ such that
$\alpha\circ\gamma=\varphi$.
The embedding problem $EP$
is {\bf locally solvable}
if,  writing $\mathbf{G}=(G,\mcG_0,\mathcal{G}_\mfp)_{\mfpinS}$ and $\mathbf{B}=(B,\mcB_0,\mathcal{B}_\mfp)_{\mfpinS}$, 
the following holds for every $\mfp\in \{0\}\cup\mfS$:
{\renewcommand{\theequation}{$LS$}
\begin{equation}\label{cond}
\begin{minipage}{13,5cm}
{\it  For every $\Gamma\in\mathcal{G}_\mfp$ there is a $\Delta\in\mathcal{B}_\mfp$,
	and for every $\Delta\in\mathcal{B}_\mfp$ there is a $\Gamma\in\mathcal{G}_\mfp$,
	such that there exists an epimorphism $\gamma_\Gamma\function\Gamma\rightarrow\Delta$
	with $\alpha\circ\gamma_\Gamma=\varphi|_\Gamma$.}
\end{minipage}
\end{equation}}

\end{Definition}

\begin{Lemma}\label{locallyzero}
If there exists $G_0\in\mcG_0$ and $B_0\in\mcB_0$ and
an epimorphism $\gamma_0\function G_0\rightarrow B_0$
with $\alpha\circ\gamma_0=\varphi|_{G_0}$,
then {\rm (\ref{cond})} holds for $\mfp=0$.
\end{Lemma}

\begin{proof}
If $g\in G$ and $\Gamma=(G_0)^g\in\mcG_0$, choose $b\in B$ with $\alpha(b)=\varphi(g)$, and let 
$\Delta=(B_0)^b$.
If $b\in B$ and $\Delta=(B_0)^b\in\mcB_0$, choose $g\in G$ with $\varphi(g)=\alpha(b)$, and let
$\Gamma=(G_0)^g$.
Define $\gamma_\Gamma\function\Gamma\rightarrow\Delta$ by $\gamma_\Gamma(x)= \gamma_0(x^{{g^{-1}}})^b$.
Then $\alpha(\gamma_\Gamma(x))=\varphi(x^{{g^{-1}}})^{\varphi(g)}=\varphi(x)$ for all $x\in\Gamma$.
\end{proof}

\begin{Lemma}\label{locallyp}
Every rigid embedding problem satisfies {\rm (\ref{cond})} for every $\mfpinS$.
\end{Lemma}

\begin{proof}
Suppose $EP$ is rigid.
If $\Gamma\in\mcG_\mfp$, choose $\Delta\in\mcB_\mfp$ with $\alpha(\Delta)=\varphi(\Gamma)$.
If $\Delta\in\mcB_\mfp$, choose $\Gamma\in\mcG_\mfp$ with $\varphi(\Gamma)=\alpha(\Delta)$.
Since $\alpha$ is rigid, $\gamma_\Gamma=(\alpha|_\Delta)^{-1}\circ\varphi|_\Gamma$ maps $\Gamma$ onto $\Delta$
and satisfies $\alpha\circ\gamma_\Gamma=\varphi|_\Gamma$.
\end{proof}

\begin{Proposition}\label{rigidlocally}
Every rigid deficient embedding problem is locally solvable.
\end{Proposition}

\begin{proof}
Suppose $EP$ is rigid and deficient.
By Lemma \ref{locallyp}, {\rm (\ref{cond})} holds for $\mfpinS$.
Since $\mbB$ is deficient, so is $\mbA$,
hence if $G_0\in\mcG_0$, then $\varphi(G_0)=1$.
Thus, {\rm (\ref{cond})} is satisfied for $\mfp=0$.
\end{proof}

\begin{Definition}
Let $\varphi\colon\mbG\rightarrow\mbA$ and $\alpha\colon\mbB\rightarrow\mbA$
be homomorphisms of deficient group piles.
Define the (asymmetric) {\bf rigid product} of $B$ and $G$ over $A$
as 
$\mathbf{B}\times_{\mathbf{A}}^{\rm rig}\mathbf{G}=(H,\mathcal{H}_\mfp)_{\mfpinS}$,
where $H=B\times_AG$ is the fiber product 
and
$$
\mathcal{H}_\mfp=\left\{ \Gamma\in{\rm Subgr}(H) \colon \beta(\Gamma)\in\mathcal{G}_\mfp,\,\pi(\Gamma)\in\mathcal{B}_\mfp,\beta|_\Gamma\mbox{  is injective} \right\},
$$
with $\pi:H\rightarrow B$ and $\beta:H\rightarrow G$ the natural projection maps.
\end{Definition}

\begin{Lemma}\label{fiberproduct}
Let $EP=(\varphi\function\mathbf{G}\rightarrow\mathbf{A},\;\alpha\function\mathbf{B}\rightarrow\mathbf{A})$
be a locally solvable embedding problem of finite deficient group piles.
Then the $\mathbf{B}\times_{\mathbf{A}}^{\rm rig}\mathbf{G}$ is a deficient group pile,
$\pi:\mathbf{B}\times_{\mathbf{A}}^{\rm rig}\mathbf{G}\rightarrow\mathbf{B}$ is an epimorphism and
$\beta:\mathbf{B}\times_{\mathbf{A}}^{\rm rig}\mathbf{G}\rightarrow\mathbf{G}$ is a rigid epimorphism.
\end{Lemma}

\begin{proof}
Since $\mathcal{G}_\mfp$ is $G$-invariant and $\mathcal{B}_\mfp$ is $B$-invariant,
$\mathcal{H}_\mfp$ is $H$-invariant.
Since $H$ is finite, 
$\mathcal{H}_\mfp$ is closed.
The projections $\beta$ and $\pi$ are surjective and, by the definition of $\mathcal{H}_\mfp$, homomorphisms of group piles.
The definition of $\mathcal{H}_\mfp$ also gives that $\beta$ is rigid.
Given $G_1\in\mathcal{G}_\mfp$, there is $B_1\in\mathcal{B}_\mfp$ and
an epimorphism $\gamma\function G_1\rightarrow B_1$ with $\alpha\circ\gamma=\varphi|_{G_1}$.
It defines a homomorphism $\hat{\gamma}\function G_1\rightarrow H$ with
$\beta\circ\hat{\gamma}={\rm id}_{G_1}$ and $\pi\circ\hat{\gamma}=\gamma$.
Let $H_1=\hat{\gamma}(G_1)$.
Then $\beta(H_1)=G_1\in\mathcal{G}_\mfp$ and $\pi(H_1)=\gamma(G_1)=B_1\in\mathcal{B}_\mfp$.
Furthermore, since $\beta\circ\hat{\gamma}={\rm id}_{G_1}$,
$\beta|_{H_1}$ is injective, so 
$H_1\in\mathcal{H}_\mfp$.
Similarly, given $B_1\in\mathcal{B}_\mfp$, there
is $H_1\in\mathcal{H}_\mfp$ with $\pi(H_1)=B_1$.
Therefore, $\beta$ and $\pi$ are epimorphisms of group piles.
\end{proof}

\begin{Remark}
The rigid product can be seen as a canonical version
of \cite[Lemma-Construction 4.2]{HJPd}.
One could define it as a sub-group pile of the fiber product in the category of deficient group piles,
which always exists.
\end{Remark}

\begin{Lemma}\label{locallyquotient}
Let $(\varphi\function\mathbf{G}\rightarrow\mathbf{A},\;\alpha\function\mathbf{B}\rightarrow\mathbf{A})$
be a locally solvable embedding problem.
Then for every normal subgroup $N\normal B$,
the induced embedding problem
$(\mathbf{G}\rightarrow\mathbf{A}/\alpha(N),\mathbf{B}/N\rightarrow\mathbf{A}/\alpha(N))$
is also locally solvable.
\end{Lemma}

\begin{proof}
Let $\tilde{\mathbf{B}}=(\tilde{B},\tilde{\mcB}_0,\tilde{\mathcal{B}}_\mfp)_{\mfpinS}=\mathbf{B}/N$
and
$\tilde{\mathbf{A}}=(\tilde{A},\tilde{\mcA}_0,\tilde{\mathcal{A}}_\mfp)_{\mfpinS}=$ $\mathbf{A}/\alpha(N)$,
and let $\pi\function\mathbf{B}\rightarrow\tilde{\mathbf{B}}$, $\tilde{\pi}\function\mathbf{A}\rightarrow\tilde{\mathbf{A}}$
be the quotient maps and
$\tilde{\alpha}\function\tilde{\mathbf{B}}\rightarrow\tilde{\mathbf{A}}$
the induced epimorphism.
Then $\tilde{\pi}\circ\alpha=\tilde{\alpha}\circ\pi$.
We have to prove that the embedding problem
$(\tilde{\pi}\circ\varphi,\tilde{\alpha})$ is locally solvable.

Let $\mfp\in \{0\}\cup\mfS$ and
let $\Gamma\in\mathcal{G}_\mfp$ be given.
Then there is a $\Delta\in\mathcal{B}_\mfp$ and an epimorphism $\gamma_\Gamma\function\Gamma\rightarrow\Delta$
with $\alpha\circ\gamma_\Gamma=\varphi|_\Gamma$.
Let $\Lambda=\pi(\Delta)\in\tilde{\mathcal{B}}_\mfp$. 
Then $\pi\circ\gamma_\Gamma\function\Gamma\rightarrow\Lambda$ is an epimorphism with
$\tilde{\alpha}\circ(\pi\circ\gamma_\Gamma)=\tilde{\pi}\circ\alpha\circ\gamma_\Gamma=(\tilde{\pi}\circ\varphi)|_\Gamma$.

Conversely, let $\Lambda\in\tilde{\mathcal{B}}_\mfp$ be given.
Choose $\Delta\in\mathcal{B}_\mfp$ with $\pi(\Delta)=\Lambda$.
Then there is a $\Gamma\in\mathcal{G}_\mfp$ and an epimorphism $\gamma_\Gamma\function\Gamma\rightarrow\Delta$
with $\alpha\circ\gamma_\Gamma=\varphi|_\Gamma$.
Hence, $\pi\circ\gamma_\Gamma\function\Gamma\rightarrow\Lambda$ is an epimorphism with
$\tilde{\alpha}\circ(\pi\circ\gamma_\Gamma)=\tilde{\pi}\circ\alpha\circ\gamma_\Gamma=(\tilde{\pi}\circ\varphi)|_\Gamma$.
\end{proof}

\begin{Lemma}\label{decomposition}
Let $(\varphi\function\mathbf{G}\rightarrow\mathbf{A},\;\alpha\function\mathbf{B}\rightarrow\mathbf{A})$
be a locally solvable finite embedding problem.
Then there exists an open normal subgroup $N\normal G$ with $N\leq{\rm Ker}(\varphi)$
such that
the induced embedding problem
$(\mathbf{G}/N\rightarrow\mathbf{A},\mathbf{B}\rightarrow\mathbf{A})$
is locally solvable.
\end{Lemma}

\begin{proof}
This is a special case of \cite[Lemma~4.1]{HJPd}.
\end{proof}

The following proposition is closely related to \cite[Lemma 4.3]{HJPd},
which we need to reprove because we have to take $e$-boundedness into consideration.

\begin{Proposition}\label{dominateSigma}
Let $(\varphi\function\mathbf{G}\rightarrow\mathbf{A},\;\alpha\function\mathbf{B}\rightarrow\mathbf{A})$
be a locally solvable $e$-bounded finite embedding problem where $\mathbf{G}$ is $e$-bounded.
Then it can be dominated by a rigid $e$-bounded finite embedding problem,
i.e.~there exist epimorphisms $\hat{\alpha}\function\hat{\mathbf{B}}\rightarrow\hat{\mathbf{A}}$,
$\hat{\varphi}\function\mathbf{G}\rightarrow\hat{\mathbf{A}}$,
$\hat{\beta}\function\hat{\mathbf{A}}\rightarrow\mathbf{A}$,
$\beta\function\hat{\mathbf{B}}\rightarrow\mathbf{B}$
such that
$\varphi=\hat{\beta}\circ\hat{\varphi}$ and
$\hat{\beta}\circ\hat{\alpha}=\alpha\circ\beta$,
and $(\hat{\varphi},\hat{\alpha})$ is a rigid $e$-bounded finite embedding problem.
\end{Proposition}

\begin{proof}
By Lemma~\ref{decomposition}, there are
a finite group pile $\hat{\mathbf{A}}$ and epimorphisms
$\hat{\varphi}\function\mathbf{G}\rightarrow\hat{\mathbf{A}}$,
$\hat{\beta}\function\hat{\mathbf{A}}\rightarrow\mathbf{A}$ with
$\varphi=\hat{\beta}\circ\hat{\varphi}$ such that
$(\hat{\beta},\alpha)$ is a locally solvable embedding problem.
Since $\mbG$ is $e$-bounded, also $\hat{\mbA}$ is $e$-bounded (Lemma~\ref{quotientrank}).

Let $\tilde{\mathbf{B}}=\mathbf{B}^{\rmdef}\times_{\mathbf{A}^{\rmdef}}^{\rm rig}\hat{\mathbf{A}}^{\rmdef}$ be the rigid product 
and let
$\tilde{\alpha}\function\tilde{\mathbf{B}}\rightarrow\hat{\mathbf{A}}^{\rmdef}$
and $\tilde{\beta}\function\tilde{\mathbf{B}}\rightarrow\mathbf{B}^{\rmdef}$ be the projections.
By Lemma~\ref{fiberproduct},
$\tilde{\alpha}$ is a rigid epimorphism and
$\tilde{\beta}$ is an epimorphism.
Choose $\hat{A}_0\in\hat{\mcA}_0$ and
$B_0\in{\mcB}_0$
and an epimorphism
$\gamma_0\function\hat{A}_0\rightarrow B_0$ 
with $\alpha\circ\gamma_0=\hat{\beta}|_{\hat{A}_0}$.
Then $\gamma_0$ defines a homomorphism $\hat{\gamma}_0\function\hat{A}_0\rightarrow\tilde{B}$
with $\tilde{\alpha}\circ\hat{\gamma}_0={\rm id}_{\hat{A}_0}$ and 
$\tilde{\beta}\circ\hat{\gamma}_0=\gamma_0$.
Let $\tilde{B}_0=\hat{\gamma}_0(\hat{A}_0)$ and note that
$\tilde{\alpha}(\tilde{B}_0)=\hat{A}_0$ and
$\tilde{\beta}(\tilde{B}_0)=B_0$.

We have that
$\rk(\tilde{B}_0)\leq\rk(\hat{A}_0)\leq e$, and
since $\hat{\mbA}$ and $\mbB$ are self-generated,
$\hat{A}=\left<\right.\hat{A}_0,\hat{A}^\prime\left.\right>$
and $B=\left<B_0,B^\prime\right>$.
Let $\hat{B}=\left<\right.\tilde{B}_0,\tilde{B}^\prime\left.\right>\leq\tilde{B}$
and $\hat{\mcB}_0=(\tilde{B}_0)^{\hat{B}}$.
Then $\hat{\mathbf{B}}=(\hat{B},\hat{\mcB}_0,\tilde{\mathcal{B}}_\mfp)_{\mfpinS}$ is a 
self-generated group pile, 
and $\tilde{\alpha}(\hat{B})=\left<\right.\hat{A}_0,\hat{A}^\prime\left.\right>=\hat{A}$
and $\tilde{\beta}(\hat{B})=\left<B_0,B^\prime\right>=B$
by Lemma~\ref{primesubgroups}.
Since $\hat{\mcA}_0=(\hat{A}_0)^{\hat{A}}$ and $\mcB_0=(B_0)^B$, 
$\tilde{\alpha}(\hat{\mcB}_0)=\hat{\mcA}_0$ and $\tilde{\beta}(\hat{\mcB}_0)=\mcB_0$,
so $\tilde{\alpha}|_{\hat{B}}$ and $\tilde{\beta}|_{\hat{B}}$ are epimorphisms of group piles.
Therefore, with $\hat{\alpha}=\tilde{\alpha}|_{\hat{B}}$
and $\beta=\tilde{\beta}|_{\hat{B}}$,
$(\hat{\varphi},\hat{\alpha})$ is a rigid $e$-bounded finite embedding problem
which dominates $(\varphi,\alpha)$.
\end{proof}

\section{Model theory of group piles}

\noindent
This section extends the comodel-theory of profinite groups in \cite{CherlinDriesMacintyre2}
(see also \cite{CherlinDriesMacintyre}) 
to group piles. 
A similar construction can be found in \cite{ErshovUniversal}.
We will give full definitions but focus our proofs on the necessary extensions to the classical theory.
For more on the comodel-theory of profinite groups see \cite{ChatzidakisDissertation}, \cite{ChatzidakisIwasawa}, \cite{Chatzidakisforking} or \cite{Frohn}.

\begin{Definition}
The {\bf colanguage}
 $\mathcal{L}_{{\rm co},\mfS}=\{\leq,\sqsubseteq,P,(G_n)_{n\in\mathbb{N}},(\mathcal{G}_{\mfp,n})_{\mfpinS,n\in\mathbb{N}}\}$
consists of
unary relation symbols $G_n$ ($n\in\mathbb{N}$), 
binary relation symbols $\leq$ and $\sqsubseteq$,
a ternary relation symbol $P$, 
and $n$-ary relation symbols $\mathcal{G}_{\mfp,n}$ ($\mfpinS$, $n\in\mathbb{N}$).
\end{Definition}

\begin{Definition}
To a deficient group pile $\mathbf{G}=(G,\mathcal{G}_\mfp)_\mfpinS$ 
assign an $\mathcal{L}_{{\rm co},\mfS}$-structure 
$S(\mathbf{G})=(S^\mathbf{G},\leq^\mathbf{G},\sqsubseteq^\mathbf{G},P^\mathbf{G},(G_n^\mathbf{G})_{n\in\mathbb{N}},(\mathcal{G}_{\mfp,n}^\mathbf{G})_{\mfpinS,n\in\mathbb{N}})$ as follows:
\begin{enumerate}
 \item[1.] $S^\mathbf{G}=\biguplus_N G/N$, where $N$ runs over all open normal subgroups of $G$.
 \item[2.] $x_1N_1\leq^\mathbf{G} x_2N_2$ if and only if $N_1\subseteq N_2$.
 \item[3.] $x_1N_1\sqsubseteq^\mathbf{G} x_2N_2$ if and only if $x_1N_1\subseteq x_2N_2$.
 \item[4.] $(x_1N_1,x_2N_2,x_3N_3)\in P^\mathbf{G}$ if and only if $N_1=N_2=N_3$ and $x_1x_2N_1=x_3N_1$.
 \item[5.] $xN\in G_n^\mathbf{G}$ if and only if $(G:N)\leq n$.
 \item[6.] $(x_1N_1,\dots,x_nN_n)\in\mathcal{G}^\mathbf{G}_{\mfp,n}$ if and only if $N_1=\dots=N_n\in G_n^\mathbf{G}$ and there is $\Gamma\in\mathcal{G}_\mfp$ such that
       $\Gamma N_1/N_1 = \{x_1N_1,\dots,x_nN_n\}$.
\end{enumerate}
\end{Definition}

\begin{Definition}
An $\mathcal{L}_{{\rm co},\mfS}$-structure $\mathbf{S}=(S,\leq,\sqsubseteq,P,(G_n)_{n\in\mathbb{N}},(\mathcal{G}_{\mfp,n})_{\mfpinS,n\in\mathbb{N}})$
is an {\bf inverse system (of group piles)} if the following statements hold:
\begin{enumerate}
\item $\leq$ is a preorder with a unique largest element.
\item If $\equivleq$ denotes the equivalence relation defined by $\leq$, 
      and $[x]$ denotes the equivalence class of $x\in S$ with respect to $\equivleq$,
      then the induced partial order on $\{[x]:x\in S\}=S/\hspace{-3pt}\equivleq$ is directed downwards.
\item $G_n=\{ x\in S: |[x]|\leq n \}$      
\item $(x,y,z)\in P$ implies $[x]=[y]=[z]$,
      and for each $x\in S$, $P\cap[x]^3$ is the graph of a binary operation
      making $[x]$ into a group. 
\item\label{ax:pile} If $(x_1,\dots,x_n)\in\mathcal{G}_{\mfp,n}$, then $x_1,\dots,x_n\in G_n$ and $[x_1]=\dots=[x_n]$.
      If moreover $y_1,\dots,y_n\in G_n$ and $\{y_1,\dots,y_n\}=\{x_1,\dots,x_n\}$,
      then $(y_1,\dots,y_n)\in\mathcal{G}_{\mfp,n}$.
      If $x\in G_n$, then, with $\mathcal{G}_{\mfp,x}=\{\{x_1,\dots,x_n\}\subseteq[x]:(x_1,\dots,x_n)\in\mathcal{G}_{\mfp,n}\}$,
      $\llbracket x\rrbracket=([x],\mathcal{G}_{\mfp,x})_\mfpinS$ is a (finite) deficient group pile.
\item $x \sqsubseteq y$ implies $x\leq y$,
      and for each $x,y\in S$ with $x\leq y$, $\sqsubseteq$ defines
      an epimorphism of group piles $\pi_{x,y}$ from $\llbracket x\rrbracket$ to $\llbracket y\rrbracket$, depending only on $[x]$ and $[y]$.
\item For all $x\leq y\leq z$, $\pi_{x,x}={\rm id}_{\llbracket x\rrbracket}$ and $\pi_{x,z}=\pi_{y,z}\circ\pi_{x,y}$.
\item\label{ax:complete} If $N$ is a normal subgroup of $\llbracket x\rrbracket$, then there is a unique $[y]$ such that $x\leq y$ and $N={\ker }(\pi_{x,y})$.
\item $S=\bigcup_{n\in\mathbb{N}} G_n$.
\end{enumerate}
\end{Definition}

\begin{Remark}\label{substructure}
Note that for $\mfS=\emptyset$, (\ref{ax:pile}) is vacant, and the remaining axioms
are exactly the ones used in \cite{CherlinDriesMacintyre2}.
Also note that (1)-(8) are $\mathcal{L}_{{\rm co},\mfS}$-elementary statements, but (9) is not.
\end{Remark}

\begin{Lemma}
If $\mathbf{G}$ is a deficient group pile,
then $S(\mathbf{G})$ is an inverse system.
\end{Lemma}

\begin{proof}
This can be checked directly from the definitions.
\end{proof}

\begin{Definition}
If $\varphi:\mathbf{G}\rightarrow\mathbf{H}$ is
an epimorphism of deficient group piles,
$\mathbf{G}=(G,\mathcal{G}_\mfp)_\mfpinS$,
$\mathbf{H}=(H,\mathcal{H}_\mfp)_\mfpinS$,
define a map $\varphi^*$ of $S(\mathbf{H})$ into $S(\mathbf{G})$ by
$$
 \varphi^*(hN)=g\varphi^{-1}(N)\in G/\varphi^{-1}(N),
$$
where $N\normal H$ is open,
$h\in H$, and $g\in G$ satisfies $\varphi(g)=h$.
\end{Definition}

\begin{Lemma}
The map $\varphi^*$ is an embedding of $\mathcal{L}_{{\rm co},\mfS}$-structures.
\end{Lemma}

\begin{proof}
First of all note that since $\varphi$ is a surjective homomorphism, $\varphi^*$ is injective 
and preserves the relations $\leq$, $\sqsubseteq$, $P$ and $(G_n)_{n\in\mathbb{N}}$.

Let $h_1N_1,\dots,h_nN_n\in S^\mathbf{H}$ and $\mfpinS$.
Then $(h_1N_1,\dots,h_nN_n)\in\mathcal{G}_{\mfp,n}^\mathbf{H}$
if and only if $N_1=\dots=N_n\in G_n^\mathbf{H}$ and there is $\Delta\in\mathcal{H}_\mfp$ such that
$\Delta N_1/N_1=\{h_1N_1,\dots,h_nN_n\}$.
Since $\varphi(\mathcal{G}_\mfp)=\mathcal{H}_\mfp$, this is the case if and only if
$\varphi^*(N_1)=\dots=\varphi^*(N_n)\in G_n^\mathbf{G}$
and there is $\Gamma\in\mathcal{G}_\mfp$ such that
$\Gamma \varphi^{-1}(N_1)/\varphi^{-1}(N_1) = \{ \varphi^*(h_1N_1),\dots,\varphi^*(h_nN_n) \}$.
This is equivalent to
$(\varphi^*(h_1N_1),\dots,\varphi^*(h_nN_n))\in\mathcal{G}_{\mfp,n}^\mathbf{G}$,
so $\varphi^*$ also preserves the relations $(\mathcal{G}_{\mfp,n})_{\mfpinS,n\in\mathbb{N}}$.
\end{proof}

\begin{Definition}
We assign to each inverse system
$\mathbf{S}=(S,\leq,\sqsubseteq,P,(G_n)_{n\in\mathbb{N}},(\mathcal{G}_{\mfp,n})_{\mfpinS,n\in\mathbb{N}})$
a deficient group pile
$G(\mathbf{S})$ 
as follows:
By axioms (1), (2) and (4)
the $[x]$ are a family of groups,
which by (3) and (9) are all finite.
By (5), the $\llbracket x\rrbracket$ are finite groups piles.
By axioms (6) and (7), the maps $\pi_{x,y}$ turns these group piles into an inverse system in the category of group piles with epimorphisms.
Let $G(\mathbf{S})=(G^\mathbf{S},\mathcal{G}_\mfp^\mathbf{S})_\mfpinS:=\varprojlim_{x\in S/\equivleq}\llbracket x\rrbracket$
be the inverse limit, cf.~Lemma \ref{lem:invlimgp}.
\end{Definition}

\begin{Lemma}
The maps $S$ and $G$ are quasi-inverse to each other.
\end{Lemma}

\begin{proof}
Let $\mathbf{G}=(G,\mathcal{G}_\mfp)_\mfpinS$ be a deficient group pile.
Since $\mathcal{G}_\mfp$ is closed,
$\mathcal{G}_\mfp=\varprojlim_N \mathcal{G}_{\mfp,N}$,
where $N$ runs over all open normal subgroups of $G$.
Therefore $G(S(\mathbf{G}))\cong \mathbf{G}$.

Conversely, let $\mathbf{S}$ be an inverse system.
Given $x\in S$, let $N_x$ be the kernel of the natural projection
$G^\mathbf{S}\rightarrow\llbracket x\rrbracket$.
Define $\psi:\mathbf{S}\rightarrow S(G(\mathbf{S}))$ to
send $x$ to its image under the natural isomorphism 
$\llbracket x\rrbracket\cong G^\mathbf{S}/N_x\subseteq S^{G(\mathbf{S})}$.
Then $\psi$ is injective and preserves
the relations $\leq$, $\sqsubseteq$, $P$, $G_n$ and $\mathcal{G}_{\mfp,n}$.
By axiom (\ref{ax:complete}), $\psi$ is surjective.
Thus $\psi$ gives an isomorphism $\mathbf{S}\cong S(G(\mathbf{S}))$.
\end{proof}

\begin{Remark}
Given the preceding lemma,
we will sometimes identify 
an inverse system $\mathbf{S}$ with $S(G(\mathbf{S}))$.
In particular,
we will treat elements of $\mathbf{S}$
as cosets $xN$ of open normal subgroups of the group pile $G(\mathbf{S})$.
\end{Remark}

\begin{Definition}
If $\psi:\mathbf{T}\rightarrow\mathbf{S}$ is an embedding of inverse systems,
define an epimorphism of profinite groups $\psi^*:G^\mathbf{S}\rightarrow G^\mathbf{T}$ as follows:

Since $\psi$ preserves the relation $\leq$, it also preserves
the relation $\equivleq$, i.e.~$\psi([x])\leq[\psi(x)]$.
Since $\psi$ preserves the relations $G_n$ ($n\in\mathbb{N}$),
$\psi$ gives a bijection $[x]\rightarrow[\psi(x)]$.
Since $\psi$ preserves the relation $P$, this bijection is an isomorphism of groups.
Since $\psi$ preserves the relation $\sqsubseteq$,
the inverse system of finite groups $([\psi(x)])_{x\in T/\equivleq}$
with homomorphisms $\pi_{\psi(x),\psi(y)}$
is isomorphic to the inverse system of finite groups 
$([x])_{x\in T/\equivleq}$
with homomorphisms $\pi_{x,y}$.
This gives a natural epimorphism
$$
 \psi^*: G^\mathbf{S}=\varprojlim_{y\in S/\equivleq}[x]\rightarrow\varprojlim_{x\in T/\equivleq}[\psi(x)]\cong\varprojlim_{x\in T/\equivleq}[x]= G^\mathbf{T}.
$$
\end{Definition}

\begin{Lemma}
The map $\psi^*$ extends to an epimorphism
$\psi^*:G(\mathbf{S})\rightarrow G(\mathbf{T})$ of group piles.
\end{Lemma}

\begin{proof}
Let $G(\mathbf{S})=(G,\mathcal{G}_\mfp)_\mfpinS$, $G(\mathbf{T})=(H,\mathcal{H}_\mfp)_\mfpinS$, 
$\mfpinS$ and $\Gamma\in\mathcal{G}_\mfp$.
For an open normal subgroup $M$ of $H$ one may check that
$\psi(\psi^*(\Gamma)M/M)=\Gamma N/N\in\mathcal{G}_{\mfp,N}$, where $\psi(H/M)=G/N$.
Since $\psi$ preserves the relations $\mathcal{G}_{\mfp,n}$, this implies $\psi^*(\Gamma)M/M\in\mathcal{H}_{\mfp,M}$.
So $\psi^*(\Gamma)\in\varprojlim_N \mathcal{H}_{\mfp,N}=\mathcal{H}_\mfp$.

Conversely, let $\Gamma\in\mathcal{H}_\mfp$ be given.
The family $\psi(\Gamma M/M)\in\mathcal{G}_{\mfp,N}$, $M\lhd H$ open and $\psi(H/M)=G/N$, is compatible with respect to the $\pi_{x,y}$, hence there exists
$\Delta\in\mathcal{G}_\mfp$ with $\psi(\Gamma M/M)=\Delta N/N$ for all $M\lhd H$ open, where $\psi(H/M)=G/N$,
and thus $\psi^*(\Delta)=\Gamma$.
Thus $\psi^*(\mathcal{G}_\mfp)=\mathcal{H}_\mfp$ for all $\mfpinS$,
so $\psi^*$ is an epimorphism of group piles.
\end{proof}

\begin{Proposition}
The map $S$ defines an equivalence of categories between the category of deficient group piles (with epimorphisms)
and the category of inverse system (with embeddings), with quasi-inverse $G$.
\end{Proposition}

\begin{proof}
Note that $S$ and $G$ are in fact functorial.
Now apply the previous lemmas.
\end{proof}

\begin{Definition}
A {\bf coformula} is an
$\mathcal{L}_{{\rm co},\mfS}$-formula that is bounded,
i.e.~all quantifiers are of the form $(\exists v\in G_n)$.
A {\bf cosentence} is a coformula without free variables.

A group pile $\mathbf{G}$ {\bf cosatisfies} 
a set $\Sigma$ of coformulas with free variables $V$ (or is a {\bf comodel} of $\Sigma$)
if there are elements $x_v\in S(\mathbf{G})$, $v\in V$, 
such that $S(\mathbf{G})\models\varphi(x_v,v\in V)$ for all $\varphi\in\Sigma$.
A {\bf cotheory} $T$ is a set of cosentences. 

The {\bf cocardinality} of a group pile $\mathbf{G}$
is the cardinality of $S(\mathbf{G})$.
A set $\Sigma$ of coformulas with parameters in some inverse system $\mathbf{S}$ is {\bf ranked}
if for every variable $v$ that occurs in some formula $\varphi\in\Sigma$,
also $G_n(v)\in\Sigma$ for some $n\in\mathbb{N}$.
A group pile $\mathbf{G}$ is {\bf $\kappa$-cosaturated}
if every ranked set $\Sigma$ of coformulas with parameters in $S(\mathbf{G})$
with $|\Sigma|<\kappa$
is cosatisfied in $\mathbf{G}$,
provided that every finite subset of $\Sigma$
is cosatisfied in $\mathbf{G}$.
\end{Definition}

\begin{Remark}
An inverse system $\mathbf{S}$ in the language $\mathcal{L}_{{\rm co},\mfS}$  
can also be viewed as an $\omega$-sorted structure $\mathbf{S}^\omega$ in a language $\mathcal{L}_{{\rm co},\mfS}^\omega$,
where the $n$-th sort consists of the $s\in G_n\setminus G_{n-1}$, 
and for each $k$-ary relation relation $R$ in $\mathcal{L}_{{\rm co},\mfS}$ we have an $\omega^k$-family of $k$-ary relations $R^{n_1,\dots,n_k}$ in $\mathcal{L}_{{\rm co},\mfS}^\omega$,
cf.~\cite[\S1]{ChatzidakisIwasawa}.
Clearly, every $\mathcal{L}_{{\rm co},\mfS}$-formula can be translated to a corresponding $\mathcal{L}_{{\rm co},\mfS}^\omega$-formula,
and vice versa.
\end{Remark}

\section{$e$-free C-piles}

\noindent
In this section, we 
generalize the {\em Cantor group piles} of \cite{HJPd}.

\begin{Definition}\label{DefCantorSigma}
An {\bf $e$-free C-pile} is an 
$e$-generated deficient group pile $\mathbf{G}$ 
for which every rigid $e$-generated deficient finite embedding problem is solvable
(cf.~Def.~\ref{defebounded}).
\end{Definition}

\begin{Lemma}\label{Cpilefree}
If $\mathbf{G}$ is an $e$-free C-pile, then $\bar{G}\cong\hat{F}_e$.
\end{Lemma}

\begin{proof}
If $B$ is a finite group with $\rk(B)\leq e$, then
$(\mathbf{G}\rightarrow 1,B\rightarrow 1)$ is a rigid $e$-generated deficient finite embedding problem for $\mathbf{G}$,
so it has a solution by assumption. 
Therefore, by Lemma~\ref{barecorrespondence},
every finite group $B$ with $\rk(B)\leq e$ is a quotient of $\bar{G}$.
Since $\rk(\bar{G})\leq e$, this implies that $\bar{G}\cong\hat{F}_e$,
cf.~\cite[16.10.7]{FJ3}.
\end{proof}

\begin{Lemma}\label{freefully}
Let $\mbG$ be an $e$-bounded 
group pile, and let
$(\varphi\function\mathbf{G}^{\rmdef}\rightarrow\tilde{\mathbf{A}},\;\alpha\function\tilde{\mathbf{B}}\rightarrow\tilde{\mathbf{A}})$
be a locally solvable $e$-generated deficient finite embedding problem.
If $G_0\cong\hat{F}_e$ for $G_0\in\mcG_0$, then 
there exist $\mbA$ and $\mbB$ with $\mbA^{\rmdef}=\tilde{\mbA}$ and $\mbB^{\rmdef}=\tilde{\mbB}$
such that
$(\varphi\function\mathbf{G}\rightarrow{\mathbf{A}},\;\alpha\function{\mathbf{B}}\rightarrow{\mathbf{A}})$
is a locally solvable $e$-bounded 
finite embedding problem.
\end{Lemma}

\begin{proof}
Let $G_0\in\mcG_0$ and $A_0=\varphi(G_0)$.
Then $G=\left<G_0,G^\prime\right>$ implies
$\tilde{A}=\left<\right.A_0,\tilde{A}^\prime\left.\right>$
(Lemma \ref{primesubgroups}),
so $\mbA=(\tilde{A},(A_0)^{\tilde{A}},\tilde{\mcA}_\mfp)_\mfpinS$
is $e$-bounded. 
By Lemma~\ref{GaschuetzPiles}, there exists
an $e$-bounded 
group pile
$\mbB=(\tilde{B},(B_0)^{\tilde{B}},\tilde{\mcB}_\mfp)_\mfpinS$
such that $\alpha\function\mbB\rightarrow\mbA$ is an epimorphism.
Without loss of generality assume that $\alpha(B_0)=A_0$.
We claim that
$EP=(\varphi\function{\mathbf{G}}\rightarrow{\mathbf{A}},\;\alpha\function{\mathbf{B}}\rightarrow{\mathbf{A}})$
is locally solvable.
Clearly it satisfies (\ref{cond}) for $\mfpinS$.
Since $G_0\cong\hat{F}_e$ and $\mbB$ is $e$-bounded,
there exists an epimorphism $\gamma_0\function G_0\rightarrow B_0$ with $\alpha\circ\gamma_0=\varphi|_{G_0}$,
cf.~\cite[17.7.3]{FJ3}.
Thus, by Lemma~\ref{locallyzero}, $EP$ satisfies (\ref{cond}) for $\mfp=0$.
\end{proof}

\begin{Proposition}\label{locallyCantor}
Every locally solvable $e$-generated deficient finite embedding problem for an $e$-free C-pile $\mbG$
is solvable.
\end{Proposition}

\begin{proof}
Let $EP=(\varphi\function\mathbf{G}\rightarrow\tilde{\mathbf{A}},\;\alpha\function\tilde{\mathbf{B}}\rightarrow\tilde{\mathbf{A}})$
be a locally solvable $e$-generated deficient finite embedding problem for $\mbG$.
By Lemma~\ref{Cpilefree},
$\bar{G}\cong\hat{F}_e$.
Let $G_0\leq G$ be a subgroup of rank at most $e$ 
that under the quotient map $G\rightarrow\bar{G}$ maps onto $\bar{G}\cong\hat{F}_e$.
Since every finite group generated by $e$ elements is a quotient of $\hat{F}_e$,
it is also a quotient of $G_0$, and thus $G_0\cong\hat{F}_e$,
cf.~\cite[16.10.7]{FJ3}.
Moreover, $\mbG^*=(G,(G_0)^G,\mcG_\mfp)_\mfpinS$ is $e$-bounded.
By Lemma~\ref{freefully},
there exist $\mbA$ and $\mbB$ with $\mbA^{\rmdef}=\tilde{\mbA}$ and $\mbB^{\rmdef}=\tilde{\mbB}$
such that $EP_1=(\varphi\function\mathbf{G}^*\rightarrow{\mathbf{A}},\;\alpha\function{\mathbf{B}}\rightarrow{\mathbf{A}})$
is locally solvable and $e$-bounded.
By Proposition~\ref{dominateSigma},
$EP_1$ can be dominated by a rigid $e$-bounded finite embedding problem $EP_2$.
The deficient reduct of $EP_2$ is a rigid $e$-generated deficient finite embedding problem, hence has a solution.
It induces a solution of $EP$.
\end{proof}

\begin{Example}
For each $\mfpinS$, let $\Gamma_\mfp$ be a profinite group and $T_\mfp$ a profinite space,
and let $\Gamma_0=\hat{F}_e$ be the free profinite group of rank $e$.
Then \cite[\S5]{HJPd} constructs from this data a certain group pile $\mathbf{G}_T$, which we call
the {\bf $e$-free semi-constant group pile} of $(\Gamma_\mfp)_\mfpinS$ over $(T_\mfp)_\mfpinS$.
We do not repeat this definition but rely on the properties of $\mathbf{G}_T$ proven in \cite{HJPd}.
\end{Example}

\begin{Lemma}\label{efscfgp}
The $e$-free semi-constant group pile $\mathbf{G}$ of 
$(\Gamma_\mfp)_\mfpinS$ over $(T_\mfp)_\mfpinS$
is an $e$-bounded self-generated 
group pile.
\end{Lemma}

\begin{proof}
By \cite[Proposition 5.3(c)]{HJPd}, $\mathbf{G}$ is self-generated.
By the construction, every $G_0\in\mathcal{G}_0$ is isomorphic to $\Gamma_0=\hat{F}_e$, hence $\mathbf{G}$
is $e$-bounded.
\end{proof}

\begin{Proposition}\label{constantCantor}
Let $\mbG$ be an $e$-free semi-constant group pile of non-trivial profinite groups $(\Gamma_\mfp)_\mfpinS$ over
perfect profinite spaces $(T_\mfp)_\mfpinS$.
Then the deficient reduct $\mbG^{\rmdef}$ of $\mbG$ is an $e$-free C-pile.
\end{Proposition}

\begin{proof}
By Lemma~\ref{efscfgp}, $\mbG$ is $e$-bounded, 
so $\mathbf{G}^{\rmdef}$ is $e$-generated.
Let $EP=(\varphi\function\mbG^{\rmdef}\rightarrow\tilde{\mbA},\alpha\function\tilde{\mbB}\rightarrow\tilde{\mbA})$ be a
rigid $e$-generated deficient finite embedding problem for $\mbG^{\rmdef}$.
By Lemma~\ref{rigidlocally}, $EP$ is locally solvable.
By Lemma~\ref{freefully}
there exist $\mbA$ and $\mbB$ with $\mbA^{\rmdef}=\tilde{\mbA}$ and $\mbB^{\rmdef}=\tilde{\mbB}$
such that
$(\varphi\function\mathbf{G}\rightarrow{\mathbf{A}},\;\alpha\function{\mathbf{B}}\rightarrow{\mathbf{A}})$
is a locally solvable $e$-bounded (and hence self-generated) embedding problem.
By \cite[Proposition 5.3(h)]{HJPd}, 
this embedding problem has a solution,
which in turn induces a solution of $EP$.
\end{proof}

\begin{Definition}\label{DefTCantor}
Let the cotheory $T_{{\rm C},\mfS,e}^{\rm co}$ consist of:
\begin{enumerate}
 \item For $n\in\mathbb{N}$ a cosentence about a group pile $\mathbf{G}=(G,\mathcal{G}_\mfp)_\mfpinS$
       stating that for each $N\lhd G$ with $(G:N)\leq n$, the finite quotient $\mathbf{G}/N$ is $e$-generated.
\item For $n,k\in\mathbb{N}$ a cosentence about a group pile $\mathbf{G}=(G,\mathcal{G}_\mfp)_\mfpinS$
       stating that for every $N\lhd G$ with $(G:N)\leq n$ and every
        rigid epimorphism $\alpha\colon\mathbf{B}\rightarrow\mathbf{G}/N$ with $\mathbf{B}$ 
        an $e$-generated deficient group pile of order $k$,
        there is an $M\lhd G$ with $(G:M)\leq k$ and $M\leq N$ and an isomorphism
        $\beta:\mathbf{G}/M\rightarrow\mathbf{B}$ such that $\alpha\circ\beta$
        is the natural map $\mathbf{G}/M\rightarrow\mathbf{G}/N$.
\end{enumerate}
\end{Definition}

\begin{Proposition}\label{AxCantor}
A deficient group pile $\mathbf{G}$ 
is an $e$-free C-pile if and only if it cosatisfies $T_{{\rm C},\mfS,e}^{\rm co}$.
\end{Proposition}

\begin{proof}
By Proposition \ref{egeneratedcharacterization}, $\mbG$ cosatisfies (1) if and only if $\mbG$ is $e$-generated.
And (2) just says that all rigid $e$-generated deficient finite embedding problems for $\mbG$ are solvable.
\end{proof}

\begin{Proposition}\label{Cantorsaturated}
Let $\mbG$ be an $\aleph_1$-cosaturated $e$-free C-pile.
Then every rigid $e$-generated deficient embedding problem $(\varphi\colon\mbG\rightarrow\mbA,\alpha\colon\mbB\rightarrow\mbA)$
with $\rk(\mbB)\leq\aleph_0$  is solvable.
\end{Proposition}

\begin{proof}
Since $\rk(B)\leq\aleph_0$, there
is a descending sequence of open normal subgroups $N_i\normal B$, $i\in\mathbb{N}$,
with $\bigcap_{i\in\mathbb{N}} N_i=1$, cf.~\cite[17.1.7(a)]{FJ3}.
For each $i\in\mathbb{N}$ let
$\alpha_i\function\mathbf{B}/N_i\rightarrow\mathbf{A}/\alpha(N_i)$
be the epimorphism induced by $\alpha$,
and 
for $i\leq j\in\mathbb{N}$ let
$\pi_i\function\mathbf{A}\rightarrow\mathbf{A}/\alpha(N_i)$,
$\rho_i\function\mathbf{B}\rightarrow\mathbf{B}/N_i$,
$\rho_{ji}\function\mathbf{B}/N_j\rightarrow\mathbf{B}/N_i$
be the quotient maps.
Then $\alpha_i\circ\rho_i=\pi_i\circ\alpha$.
\begin{diagram}
 && \mbG \\
 &\ldDotsto(2,6)^{\gamma_i}& \dTo_\varphi \\
\mbB &\rTo^\alpha & \mbA \\
\dTo^{\rho_j} && \dTo_{\pi_j} \\
\mbB/N_j &\rTo^{\alpha_j} &\mbA/\alpha(N_j) \\
\dTo^{\rho_{ji}} && \dTo \\
\mbB/N_i &\rTo^{\alpha_i} &\mbA/\alpha(N_i) \\
\end{diagram}
By Lemma~\ref{rigidlocally},
the rigid deficient embedding problem $(\varphi,\alpha)$
is locally solvable,
hence the induced embedding problem $(\pi_i\circ\varphi,\alpha_i)$
is locally solvable by Lemma~\ref{locallyquotient}.
Since $\mbB$ is $e$-generated, $\mbB/N_i$ is $e$-generated by Lemma~\ref{quotientrank}.
Hence,
$(\pi_i\circ\varphi,\alpha_i)$ is a locally solvable $e$-generated deficient finite embedding problem for $\mathbf{G}$.
Since $\mbG$ is an $e$-free C-pile,
this embedding problem has a solution
$\gamma_i\function\mathbf{G}\rightarrow\mathbf{B}/N_i$
by Proposition~\ref{locallyCantor}.

For each $i$, fix an enumeration $\mathbf{B}/N_i=\{b_{i,1},\dots,b_{i,n_i}\}$ and let $a_{i,\nu}=\alpha_i(b_{i,\nu})\in\mbA/\alpha(N_i)\subseteq S(\mbA)$. 
View $S(\mbA)$ as a subset of $S(\mbG)$ via $\varphi^*$ and
let $\Sigma$ be the following set of bounded $\mathcal{L}_{{\rm co},\mfS}$-formulas in the variables $x_{i,\nu}$, $i\in\mathbb{N}$, $1\leq \nu\leq n_i$, with constants from $S(\mbA)$:
\begin{enumerate}
\item For each $i$ and each $1\leq \nu\leq n_i$ the $\mathcal{L}_{{\rm co},\mfS}$-formula $G_{n_i}(x_{i,\nu})$.
\item For each $i$ an $\mathcal{L}_{{\rm co},\mfS}$-formula stating that $[x_{i,1}]=\{x_{i,1},\dots,x_{i,n_i}\}$ and that
  the map $x_{i,\nu}\mapsto b_{i,\nu}$ is an isomorphism of group piles $\beta_i:\llbracket x_{i,1} \rrbracket\rightarrow \mathbf{B}/N_i$.
\item For each $i$ an $\mathcal{L}_{{\rm co},\mfS}$-formula with constants from $S(\mbA)$ stating that $x_{i,1}\leq a_{i,1}$ and
       $\pi_{x_{i,1},a_{i,1}}(x_{i,\nu})=a_{i,\nu}$ for all $1\leq\nu\leq n_i$.
\item For each $i\leq j$, an $\mathcal{L}_{{\rm co},\mfS}$-formula stating that 
 $x_{j,1}\leq x_{i,1}$ and $\beta_i\circ\pi_{x_{j,1}x_{i,1}}=\rho_{ji}\circ\beta_j$.
\end{enumerate}
Every finite subset $\Sigma_0$ of $\Sigma$ is cosatisfied in $\mathbf{G}$:
Let $j$ be the maximal index of a variable $x_{j,\nu}$ appearing in $\Sigma$ and for $i\leq j$, $1\leq\nu\leq n_i$ let 
$g_{i,\nu}=\gamma_j^*(\rho_{ji}^*(b_{i,\nu}))$.
Then $(g_{i,\nu})_{1\leq i\leq j,1\leq\nu\leq n_i}$ satisfies $\Sigma_0$.
By (1), $\Sigma$ is ranked.

Thus, since $\mathbf{G}$ is $\aleph_1$-cosaturated, there are $(\tilde{g}_{i,\nu})_{1\leq i,1\leq\nu\leq n_i}$ in $\mathbf{G}$ that satisfy $\Sigma$.
For each $i$, the map $b_{i,\nu}\mapsto\tilde{g}_{i,\nu}$ gives an isomorphism $\mathbf{B}/N_i\rightarrow\llbracket\tilde{g}_{i,1}\rrbracket$ by (2),
hence has as dual an epimorphism $\tilde{\gamma}_i\colon\mathbf{G}\rightarrow\mathbf{B}/N_i$,
which satisfies $\alpha_i\circ\tilde{\gamma}_i=\pi_i\circ\varphi$ by (3).
By (4), these epimorphisms are compatible, giving rise to an epimorphism $\tilde{\gamma}=\varprojlim_i\tilde{\gamma}_i\colon\mathbf{G}\rightarrow\varprojlim_i\mathbf{B}/N_i=\mathbf{B}$,
which then satisfies $\alpha\circ\tilde{\gamma}=\varphi$.
\end{proof}

\section{Model theory of $\PSCC$ fields}

\noindent
In the next section we will let the finite set $\mfS$ be a set of {\em primes} and associate to each field $F$
a group pile $\GalS(F)$, which extends the absolute Galois group $\Gal(F)$ with $\mfS$-local data.
The notion of {\em prime} and the corresponding local-global principle $\PSCC$ we use is the one developed in \cite{AFPSCC}.
We now briefly recall the main definitions and results but refer to \cite[\S2-3]{AFPSCC} for further details\footnote{The reader who want to check these details should be aware of the fact
that, in the notation of \cite{AFPSCC}, here we consider only the case of relative type $\tau=(1,1)$, so for example the $\PSCC$ property is there called ${\rm P}S^\tau{\rm CC}$ with $S=\mfS$ and $\tau=(1,1)$.}.
Basics on real closed and $p$-adically closed fields are summarized in Appendix \ref{app:real} and \ref{app:padic}.

\begin{Definition}
For a field $F$ of characteristic zero
we denote by $\tilde{F}$ a fixed algebraic closure of $F$, and by $\Gal(F)=\Gal(\tilde{F}/F)$ the absolute Galois group of $F$.
\end{Definition}

\begin{Definition}
A {\bf prime}\footnote{called a {\em classical prime} in \cite{AFPSCC}} of a field $K$ is either an ordering of $K$ or an equivalence class of $p$-valuations on $K$, for some prime number $p$.
It is {\bf local} if the ordering is archimedean resp.~if the value group is isomorphic to $\mathbb{Z}$.
If $\mfP$ is a prime of $K$, we denote by $\CC(K,\mfP)$ the set of all real resp.~$p$-adic closures of $(K,\mfP)$ inside $\tilde{K}$.
If $\mfp$ is a prime of $K$ and $F/K$ is a field extension, 
we denote by $\mcS_\mfp(F)$ the set of all primes $\mfP$ of $F$ that lie above $\mfp$ (we write this as $\mfP|_K=\mfp$) and are of the same type,
and by $\CC_\mfp(F)$ the union of all $\CC(F,\mfP)$, $\mfP\in\mcS_\mfp(F)$,
cf.~\cite[Def.~3.4, Def.~3.7, Def.~4.2]{AFPSCC}.
\end{Definition}

\begin{Setting}
For the rest of this work let $\mfS$ be a finite set of local primes of 
a field $K$ of characteristic zero, and let $F/K$ be a field extension.
For $\mfp\in\mfS$ fix a closure $K_\mfp\in\CC_\mfp(K)$.
\end{Setting}

\begin{Lemma}\label{CCsubfield}
Let $K\subseteq E\subseteq F$,
$\mfpinS$, $\mfQ\in\mcS_\mfp(F)$ and $\mfP=\mfQ|_E\in\mcS_\mfp(E)$.
If $F^\prime\in\CC(F,\mfQ)$, then $E^\prime:=F^\prime\cap\tilde{E}\in\CC(E,\mfP)$
and ${\rm res}\function\Gal(F^\prime)\rightarrow\Gal(E^\prime)$ is an isomorphism.
In particular, $F^\prime\cap\tilde{E}\in\CC_\mfp(E)$ for any $F^\prime\in\CC_\mfp(F)$.
\end{Lemma}

\begin{proof}
The field $E^\prime$ is algebraically closed
in the real closed resp.~$p$-adically closed field $F^\prime$,
so it is real closed resp.~$p$-adically closed itself, see Lemma~\ref{algclosedreal} and Lemma~\ref{algclosedpadic}.
Let $\mfQ^\prime$ be the unique prime of $F^\prime$ over $\mfQ$.
Then $\mfP^\prime=\mfQ^\prime|_{E^\prime}$ is the unique prime of $E^\prime$
of the same type as $\mfp$,
so $E^\prime\in\CC(E,\mfP^\prime|_E)$.
Since $\mfP^\prime|_E=\mfQ^\prime|_E=\mfP$,
it follows that
$E^\prime\in\CC(E,\mfP)$.

Since $E^\prime\equiv F^\prime$ by model completeness (Proposition~\ref{QERCF} and Proposition~\ref{QEpCF}),
and $\Gal(F^\prime)$ is finitely generated (Proposition~\ref{ArtinSchreier} and Proposition~\ref{Galpadic}),
$\Gal(E^\prime)\cong\Gal(F^\prime)$ by \cite[20.4.6]{FJ3}.
Thus the epimorphism ${\rm res}\function\Gal(F^\prime)\rightarrow\Gal(E^\prime)$ is an isomorphism by \cite[16.10.8]{FJ3}.
\end{proof}

\begin{Definition}
We say that $F$ is $\mathbf{P}\mfS\mathbf{CC}$ 
if it satisfies the following local-global principle
for any smooth absolutely irreducible $F$-variety $V$:
$V(F)\neq\emptyset$ iff $V(F')\neq\emptyset$ for all $F'\in\CC_\mfp(F)$, $\mfpinS$.
\end{Definition}

\begin{Definition}
A prime $\mfP$ is {\bf quasi-local} if it is an ordering or
a $p$-valuation with value group a $\mathbb{Z}$-group, cf.~\cite[Def.~3.5, Rem.~3.6]{AFPSCC},
and $F$ is {\bf $\mfS$-quasi-local} if all $\mfP\in\mcS_\mfp(F)$, $\mfpinS$, are quasi-local.
The field $F$ is {\bf $\mfS$-SAP} if it satisfies the strong approximation property of \cite[Def.~10.1]{AFPSCC}.
(We will not make use of the precise definition.)
\end{Definition}

\begin{Proposition}\label{PSCC:SAP}\label{PSCC:quasilocal}
If $F/K$ is algebraic, or $F$ is $\PSCC$, then $F$ is $\mfS$-quasi-local and $\mfS$-SAP.
\end{Proposition}

\begin{proof}
See \cite[Lemma 4.8, Prop.~4.9, Lemma 10.5, Prop.~10.7]{AFPSCC}.
\end{proof}

\begin{Definition}
An extension $M/F$ is {\bf totally $\mfS$-adic} if the restriction map $\mcS_\mfp(M)\rightarrow\mcS_\mfp(F)$, $\mfP\mapsto\mfP|_F$,
is surjective for all $\mfpinS$, cf.~\cite[Def.~11.1]{AFPSCC}.
For $\mfpinS$ we let $R_\mfp(F)=\bigcap_{\mfP\in\mcS_\mfp(F)}\mathcal{O}_\mfP$,
where $\mathcal{O}_\mfP$ is the positive cone resp.~the valuation ring of $\mfP$, cf.~\cite[Def.~4.2]{AFPSCC}.
\end{Definition}

\begin{Proposition}\label{totallySadic}
If $F$ is $\mfS$-SAP,
then the following statements are equivalent
for 
every extension $M/F$:
\begin{enumerate}
 \item $M/F$ is totally $\mfS$-adic.
 \item $R_\mfp(M)\cap F=R_\mfp(F)$ for every $\mfpinS$.
 \item $R_\mfp(M)\cap F\subseteq R_\mfp(F)$ for every $\mfpinS$.
\end{enumerate}
\end{Proposition}

\begin{proof}
See \cite[Lemma 11.4]{AFPSCC}.
\end{proof}

We now recall some results on the model theory of $\PSCC$ fields.

\begin{Definition}
Let $\mcLr=\{+,-,\cdot,0,1\}$ be the language of rings,
$\mathcal{L}_{{\rm ring},\mfS}=\mcLr\cup\{R_\mfp:\mfpinS\}$, where each $R_\mfp$ is a unary predicate symbol,
and $\mathcal{L}_{{\rm ring},\mfp}=\mcL_{{\rm ring},\{\mfp\}}$.
For a language $\mcL$ we denote by $\mcL(K)=\mcL\cup\{c_a:a\in K\}$ the augmentation by constants from $K$.
\end{Definition}

\begin{Proposition}\label{thm:PSCC1}
There is a recursive $\mcL_{\rm ring}(K)$-theory $T_\PSCC$
such that $F$ satisfies $T_{{\rm P}\mfS{\rm CC}}$ if and only if $F$ is $\PSCC$.
\end{Proposition}

\begin{proof}
See \cite[Prop.~9.3]{AFPSCC}.
\end{proof}

\begin{Proposition}\label{PSCC:holomorphy}
For each $\mfpinS$ there exists an $\mcLr(K)$-formula $\varphi_{R,\mfp}$ that defines
$R_\mfp(F)$ in $F$ for each $\PSCC$ field $F$.
\end{Proposition}

\begin{proof}
See \cite[Thm.~1.2]{AFPSCC}.
\end{proof}

\begin{Proposition}\label{PSCC:quantify}
For every $\mfpinS$
there exists a recursive map $\varphi(\mathbf{x})\mapsto\hat{\varphi}_{\mfp,\exists}(\mathbf{x})$  
from $\mcLr$-formulas to $\mathcal{L}_{{\rm ring},\mfp}(K)$-formulas
such that for $F\supseteq K$ which is $\mfS$-quasi-local and $a_1,\dots,a_m\in F$,
one has
$(F,R_\mfp(F))\models\hat{\varphi}_{\mfp,\exists}(\mathbf{a})$
iff
$F^\prime\models\varphi(\mathbf{a})$ for some $F^\prime\in\CC_\mfp(F)$.
\end{Proposition}

\begin{proof}
Apply \cite[Lemma 8.3]{AFPSCC} to $\neg\varphi(\mathbf{x})$.
\end{proof}

\begin{Proposition}\label{PSCC:elementaryext}
If $F$ is $\PSCC$ and $F\prec M$ is an elementary extension, 
then $M/F$ is regular and totally $\mfS$-adic.
\end{Proposition}

\begin{proof}
See \cite[Cor.~11.5]{AFPSCC}.
\end{proof}

The following embedding theorem will play a central role in Section \ref{sec:axiomatization}:

\begin{Proposition}[Pop]\label{PSCEmbedding}
Let $L\supseteq K$
and let $E/L$, $F/L$ be regular extensions, where
$E$ is countable and $F$ is $\aleph_1$-saturated and $\PSCC$.
Then for every homomorphism 
$\gamma\function{\rm Gal}(F)\rightarrow{\rm Gal}(E)$
with ${\rm res}_{\tilde{E}/\tilde{L}}\circ\gamma={\rm res}_{\tilde{F}/\tilde{L}}|_{{\rm Gal}(F)}$,
there exists an $L$-embedding 
$\tilde{E}\rightarrow \tilde{F}$ 
such that
$\gamma(\tau)=\tau|_{\tilde{E}}$ for all $\tau\in\Gal(F)$.
\end{Proposition}

\begin{proof}
This follows from \cite[6.1]{PopDissertation}, as $\PSCC$ fields are pseudo classically closed, see also \cite[Prop.~4.6]{AFPSCC}.
A proof in the present setting with all details can be found in \cite[2.11.5]{AFDiss}.
\end{proof}

\section{$\mfS$-adic Absolute Galois Group Piles}

\noindent
We now define the group pile $\GalS(F)$ and prove some of its basic properties.

\begin{Definition}\label{DefGalp}
The {\bf $\mfS$-adic absolute Galois group pile} of $F$ is the group pile
$$
 \GalS(F)=(\Gal(F), \mcG_\mfp  )_{\mfpinS},
$$ 
where $\mcG_\mfp=\{\Gal(F^\prime) \colon F^\prime\in\CC_\mfp(F)\}$.
For a Galois extension $E/F$, let
$\GalS(E/F)=\GalS(F)/\Gal(E)$ be the
{\bf $\mfS$-adic Galois group pile} of $E/F$.
\end{Definition}

In order to prove that $\GalS(F)$ is a indeed group pile, we will make use of 
the following group theoretical lemma:

\begin{Lemma}\label{fgclosed}
Let $G$ be a profinite group and $\Gamma$ a finitely generated profinite group.
Then 
$\mcG=\{ H\leq G \colon H\mbox{ is a quotient of }\Gamma \}\subseteq{\rm Subgr}(G)$ 
is closed.
\end{Lemma}

\begin{proof}
We prove that ${\rm Subgr}(G)\setminus\mcG$ is open.
Let $H\leq G$ such that $H$ is not a quotient of $\Gamma$.
Since $\Gamma$ is finitely generated,
by \cite[16.10.7(a)]{FJ3}
there exists an open normal subgroup $H_0\normal H$ such that $H/H_0$ 
is not a quotient of $\Gamma$.
Let $N\normal G$ be an open normal subgroup with $N\cap H\leq H_0$.
Since $H/H_0$ is not a quotient of $\Gamma$,
also $H/(N\cap H)$ is not a quotient of $\Gamma$.
If $H^\prime\leq G$ and $H^\prime N=H N$, then
$H^\prime/(N\cap H^\prime)\cong H^\prime N/N\cong H/(N\cap H)$,
hence $H^\prime$ is not a quotient of $\Gamma$.
The set of such $H'$ forms an open neighborhood of $H$.
\end{proof}

The following statement is similar 
to \cite[Lemma 10.3(c)-(d)]{HJPd},
which, however, is concerned with fields instead of group piles,
and is restricted to certain subfields of $\KtotS$.

\begin{Proposition}\label{SadicGal}
The $\mfS$-adic absolute Galois group pile $\GalS(F)$ is a separated reduced deficient group pile.
\end{Proposition}

\begin{proof}
Let $\mbG=(G,\mcG_\mfp)_\mfpinS=\GalS(F)$.

We first prove that $\mbG$ is a group pile.
Let
$\GalS(K)=(H,\mcH_\mfp)_\mfpinS$ 
and fix $\mfpinS$.
We have to show that $\mcG_\mfp$ is closed in ${\rm Subgr}(G)$.
Let $\Gamma=\Gal(K_\mfp)$.
Since $\mfp$ is local, we have $\mcH_\mfp=\Gamma^H$, cf.~\cite[Rem.~3.6]{AFPSCC}, hence $\mcH_\mfp$ is closed in ${\rm Subgr}(H)$.
By Proposition \ref{ArtinSchreier} and Proposition \ref{Galpadic}, $\Gamma$ is finitely generated.

Let $G_0\leq G$.
We claim that $G_0\in\mcG_\mfp$ if and only if $G_0$ is a quotient of $\Gamma$ and
${\rm res}_{\tilde{F}/\tilde{K}}(G_0)\in\mcH_\mfp$.
Indeed, if $G_0\in\mcG_\mfp$, then ${\rm res}_{\tilde{F}/\tilde{K}}(G_0)\in\mcH_\mfp$
and $G_0\cong{\rm res}_{\tilde{F}/\tilde{K}}(G_0)\cong\Gamma$ 
by Lemma~\ref{CCsubfield}.
Conversely, 
if ${\rm res}_{\tilde{F}/\tilde{K}}(G_0)\in\mcH_\mfp=\Gamma^H$,
then $\Gamma$ is quotient of $G_0$.
Hence, if also $G_0$ is a quotient of $\Gamma$,
then $G_0\cong\Gamma$ by \cite[16.10.7]{FJ3}.
Therefore, 
by Proposition \ref{ArtinSchreier} and Proposition \ref{NPEK},
the fixed field $F^\prime$ of $G_0$ is real closed resp.~$p$-adically closed
of the same type as $\mfp$.
In addition,
${\rm res}_{\tilde{F}/\tilde{K}}(G_0)\in\mcH_\mfp$
implies that $F^\prime\in\CC_\mfp(F)$, i.e.~$G_0\in\mcG_\mfp$.

By Lemma~\ref{fgclosed}, the set
of $G_0\leq G$ such that $G_0$ is a quotient of $\Gamma$ is closed.
Since $\mcH_\mfp=\Gamma^H$ is closed,
and ${\rm res}_{\tilde{F}/\tilde{K}}\function{\rm Subgr}(G)\rightarrow{\rm Subgr}(H)$
is continuous by Lemma~\ref{Subgrfunctor}, 
the set of $G_0\leq G$ with
${\rm res}_{\tilde{F}/\tilde{K}}(G_0)\in\mcH_\mfp$
is closed.
Therefore, $\mcG_\mfp$ is closed.

We now prove that $\mbG$ is separated and reduced.
Let $\mfp,\mfq\in\mfS$, $\Gamma\in\mcG_\mfp$, $\Gamma_1\in\mcG_\mfq$,
and assume that $\Gamma\subseteq\Gamma_1$.

If $\mfp$ or $\mfq$ is an ordering, 
then both are orderings and $\Gamma=\Gamma_1$,
since the absolute Galois group of a real closed field is finite (Proposition~\ref{ArtinSchreier}),
and the absolute Galois group of a $p$-adically closed field is non-trivial and torsion-free (Proposition~\ref{Galpadic}).
So since the ordering of a real closed field is unique, $\mfp=\mfq$.

If $\mfp$ is a $p$-valuation and $\mfq$ is a $q$-valuation,
let $F^\prime$ and $F_1^\prime$ be the fixed fields of $\Gamma$ resp.~$\Gamma_1$,
and let $K^\prime=\tilde{K}\cap F^\prime$ and $K_1^\prime=\tilde{K}\cap F_1^\prime$.
Then $K_1^\prime\subseteq K^\prime$, and $K^\prime\in\CC_\mfp(K)$ and $K_1^\prime\in\CC_\mfq(K)$
by Lemma~\ref{CCsubfield}.
Thus, since $\mfp$ and $\mfq$ are local,
$K^\prime$ is Henselian with respect to two rank one valuations,
which must be equivalent by F.~K.~Schmidt's theorem, cf.~\cite[4.4.1]{EnglerPrestel}. 
In particular, $p=q$.
Thus the restriction of the unique $p$-valuation of $F^\prime$ to $F_1^\prime$
is the unique $p$-valuation of $F_1^\prime$, so $\mfp=\mfq$.
Therefore, by the maximality of $p$-adically closed fields (of the same type),
$F^\prime=F_1^\prime$,
hence $\Gamma=\Gamma_1$.
\end{proof}

\begin{Remark}
If $(N_i)_{i\in I}$ is a directed family of closed normal subgroups of 
a group pile $\mathbf{G}$ with $\bigcap_{i\in I}N_i=1$, then
$\mathbf{G}\cong\varprojlim_{i\in I}\mathbf{G}/N_i$.
In particular,
$\GalS(F)=\varprojlim_E \GalS(E/F)$,
where $E$ ranges over all finite Galois extensions of $F$.
\end{Remark}

\begin{Lemma}\label{rigid}
Let $F\supseteq K$ be $\mfS$-quasi-local,
let $M/F$ be an extension, and
let
$$
 {\rm res}_{\tilde{M}/\tilde{F}}\function \GalS(M)\rightarrow\GalS(F)
$$
be the restriction map.
Then 
${\rm res}_{\tilde{M}/\tilde{F}}$ is a homomorphism
of group piles, and the following are equivalent:
\begin{enumerate}
 \item ${\rm res}_{\tilde{M}/\tilde{F}}$ is an epimorphism.
 \item ${\rm res}_{\tilde{M}/\tilde{F}}$ is a rigid epimorphism.
 \item $M/F$ is regular and totally $\mfS$-adic.
\end{enumerate}
\end{Lemma}

\begin{proof}
Let $\mfpinS$.
If $M^\prime\in{\rm CC}_\mfp(M)$, then 
$F^\prime=M^\prime\cap\tilde{F}\in{\rm CC}_\mfp(F)$ by Lemma~\ref{CCsubfield}.
Thus, ${\rm res}_{\tilde{M}/\tilde{F}}\function{\rm Gal}(M)\rightarrow{\rm Gal}(F)$
indeed induces a homomorphism of group piles
${\rm res}_{\tilde{M}/\tilde{F}}:\GalS(M)\rightarrow\GalS(F)$.

Proof of $(1)\Rightarrow(3)$:
Suppose that 
${\rm res}_{\tilde{M}/\tilde{F}}\function\GalS(M)\rightarrow\GalS(F)$
is an epimorphism of group piles.
Then
${\rm res}_{\tilde{M}/\tilde{F}}\function\Gal(M)\rightarrow\Gal(F)$
is surjective,
so $M/F$ is regular.
Let $\mfpinS$, $\mfP\in\mcS_\mfp(F)$, and $F^\prime\in{\rm CC}(F,\mfP)$.
Then there exists $M^\prime\in{\rm CC}_\mfp(M)$ with $M^\prime\cap\tilde{F}=F^\prime$.
Let $\mfQ^\prime$ be the unique prime of $M^\prime$ lying over $\mfp$
and let $\mfQ=\mfQ^\prime|_M$.
Then $\mfQ\in\mcS_\mfp(M)$ and $\mfQ|_F=\mfP$.
Therefore, $M/F$ is totally $\mfS$-adic.

Proof of $(3)\Rightarrow(2)$:
Since $M/F$ is regular, 
${\rm res}_{\tilde{M}/\tilde{F}}\function{\rm Gal}(M)\rightarrow{\rm Gal}(F)$
is surjective.
Consider $\mfpinS$, $\mfP\in\mcS_\mfp(F)$, and $F^\prime\in{\rm CC}(F,\mfP)$.
Since $M/F$ is totally $\mfS$-adic, there exists $\mfQ\in\mcS_\mfp(M)$ lying over $\mfP$.
If $M^{\prime\prime}\in{\rm CC}(M,\mfQ)$, then $F^{\prime\prime}=M^{\prime\prime}\cap\tilde{F}\in{\rm CC}(F,\mfP)$
by Lemma~\ref{CCsubfield}.
Since $\mfP$ is quasi-local,
$F^\prime$ and $F^{\prime\prime}$ are conjugate over $F$, cf.~\cite[Rem~3.6]{AFPSCC}.
Since ${\rm res}_{\tilde{M}/\tilde{F}}$
is surjective, there exists a conjugate $M^\prime\in{\rm CC}(M,\mfQ)$ of $M^{\prime\prime}$
with $M^\prime\cap\tilde{F}=F^\prime$.
Therefore, 
${\rm res}_{\tilde{M}/\tilde{F}}\function\GalS(M)\rightarrow\GalS(F)$
is an epimorphism of group piles.
By Lemma~\ref{CCsubfield}, ${\rm res}\function\Gal(M^\prime)\rightarrow\Gal(F^\prime)$
is an isomorphism, so
${\rm res}_{\tilde{M}/\tilde{F}}$ is rigid.

Proof of $(2)\Rightarrow(1)$: This is trivial.
\end{proof}

We now explain how to interpret statements about $\GalS(F)$ in $F$. 
Due to lack of a suitable reference we present the classical case $\mfS=\emptyset$ in Appendix \ref{app:absGal}.
Here we only explain how to extend this to general $\mfS$.
The following proposition generalizes Proposition \ref{prop:coring}.

\begin{Proposition}\label{prop:coringS}
To every ranked set of coformulas $\Sigma$ in the variables $v_1,v_2,\dots$ we can assign recursively
a set $\Sigma_{\rm ring}$ of $\mathcal{L}_{{\rm ring},\mfS}$-formulas such that
for every $F\supseteq K$ which is $\mfS$-quasi-local,
$\Sigma$ is cosatisfied in $\GalS(F)$ if and only if $\Sigma_{\rm ring}$ is satisfied in $(F,R_\mfp(F))_\mfpinS$.
\end{Proposition}

\begin{proof}
Building on the case $\mfS=\emptyset$,
we only have to explain how to translate statements of the form $(v_1,\dots,v_n)\in\mathcal{G}_{\mfp,n}$.
Let $\GalS(F)=(G,\mathcal{G}_\mfp)_\mfpinS$, 
$N={\rm Gal}(E)\normal G$ an open subgroup,
and $\Delta N/N\in{\rm Subgr}(G/N)$.
By definition, $\Delta N/N \in\mathcal{G}_{\mfp,N}$ if and only if there is
some $F'\in\CC_\mfp(F)$ such that,
with $\Gamma={\rm Gal}(F^\prime)\in\mathcal{G}_\mfp$,
$\Gamma N/N=\Delta N/N$.
So $(x_1N_1,\dots,x_nN_n)\in\mathcal{G}_{\mfp,n}$ if and only
if $N_1=\dots=N_n$ and $H=\{x_1N_1,\dots,x_nN_n\}$ is a subgroup of $G/N_1={\rm Gal}(E/F)$
that corresponds to a field $F\subseteq E^\prime\subseteq E$
that is the intersection of some $F'\in\CC_\mfp(F)$ with $E$.
This is equivalent to the fact
that every polynomial $f\in F[X]$ of degree bounded by $[E:F]$
that has a root in $E$
has a root in $F^\prime$ if and only if it has a root in $E^\prime$.
By Proposition \ref{PSCC:quantify}, the existence of such an $F'$ can be expressed by an $\mathcal{L}_{{\rm ring},\mfS}$-formula in $(F,R_\mfp(F))_\mfpinS$.
\end{proof}

\begin{Corollary}\label{interpretCDM}
There is a recursive map $\varphi\mapsto\varphi_{\rm ring}$ from cosentences to $\mathcal{L}_{\rm ring}(K)$-sentences
such that for every $\PSCC$ field $F$ and every cosentence $\varphi$, we have that
$F\models\varphi_{\rm ring}$ iff $S(\GalS(F))\models\varphi$.
\end{Corollary}

\begin{proof}
A $\PSCC$ field $F$ is $\mfS$-quasi-local (Proposition \ref{PSCC:quasilocal}) and the $R_\mfp(F)$ are $K$-definable (Proposition \ref{PSCC:holomorphy})
by some formula $\varphi_{R,\mfp}$,
so the claim follows from the special case $\Sigma=\{\varphi\}$ of Proposition \ref{prop:coringS}
by replacing all occurrences of the predicates $R_\mfp$ by $\varphi_{R,\mfp}$.
\end{proof}

\begin{Corollary}\label{cosaturated}
If $F$ is $\PSCC$ and $\aleph_1$-saturated,
then $\GalS(F)$ is $\aleph_1$-cosaturated.
\end{Corollary}

\begin{proof}
This follows from Proposition \ref{prop:coringS} just like Corollary \ref{cor:sat0} follows from Proposition~\ref{prop:coring},
using, like in the proof of the previous corollary, Proposition \ref{PSCC:quasilocal} and Proposition \ref{PSCC:holomorphy}.
\end{proof}

\begin{Definition}\label{DefTCantorring}
If we apply Corollary \ref{interpretCDM} to the sentences of the
cotheory $T_{{\rm C},\mfS,e}^{\rm co}$ (Definition \ref{DefTCantor}), 
we get an $\mcL_{\rm ring}(K)$-theory, which we denote by $T_{{\rm C},\mfS,e}^{\rm ring}$.
\end{Definition}

\section{Subfields of $\KtotS$}

\noindent
We now turn to the fields $\KtotS(\bfsigma)$ mentioned in the introduction.
We can define
$\KtotS=\bigcap_\mfpinS\bigcap\CC_\mfp(K)$,
cf.~\cite[0.1]{GeyerJarden}, \cite{HJPd}.

\begin{Lemma}\label{KtotSmaxtotS}\label{totSCC}\label{QtotpKochen}
Let $K\subseteq L\subseteq \KtotS$ be a field and let $\mfpinS$. Then the following holds: 
\begin{enumerate}
\item $L/K$ is totally $\mfS$-adic.
\item $\CC_\mfp(L)=\CC_\mfp(K)$
\item $R_\mfp(L)=L\cap R_\mfp(\KtotS)=L\cap\bigcap_{K'\in\CC_\mfp(K)}R_\mfp(K')$
\end{enumerate}
\end{Lemma}

\begin{proof}
(1): 
For $\mfpinS$, take $K'\in\CC_\mfp(K)$ and $\mfP\in\mcS_\mfp(K')$.
Since $L\subseteq\KtotS\subseteq K'$
we can restrict $\mfP$ to $\mfP|_L\in\mcS_\mfp(L)$, which lies over $\mfp$. 

(2): Since $L/K$ is algebraic, it is clear that $\CC_\mfp(L)\subseteq\CC_\mfp(K)$.
Conversely, if $K'\in\CC_\mfp(K)$,  
then $L\subseteq\KtotS\subseteq K'$
and therefore $K'\in\CC_\mfp(L)$.

(3): By (2), $\CC_\mfp(L)=\CC_\mfp(K)=\CC_\mfp(\KtotS)$.
Thus,
\begin{eqnarray*}
 R_\mfp(L)&=&\bigcap_{\mfP\in\mcS_\mfp(L)}\mcO_\mfP
   \;=\; \bigcap_{L^\prime\in\CC_\mfp(L)} R_\mfp(L^\prime) \cap L
   \;=\; L\cap\bigcap_{K'\in\CC_\mfp(K)}R_\mfp(K') \\
   &=& L\cap \bigcap_{K^\prime\in\CC_\mfp(\KtotS)} R_\mfp(K^\prime) \cap \KtotS 
   \;=\; L\cap \bigcap_{\mfP\in\mcS_\mfp(\KtotS)}\mcO_\mfP \;=\; L\cap R_\mfp(\KtotS).
\end{eqnarray*}
\end{proof}

\begin{Lemma}\label{PSCCextension}
Let $K\subseteq E\subseteq F\subseteq\KtotS$.
If $E$ is $\PSCC$, then $F$ is $\PSCC$.
\end{Lemma}

\begin{proof}\footnote{This proof corrects an inaccuracy in \cite[Proof of Lemma 1.6 Part B]{GeyerJarden}.}
Let $V$ be a smooth absolutely irreducible variety defined over $F$
with $V(F^\prime)\neq\emptyset$ for all $F^\prime\in\CC_\mfp(F)$, $\mfpinS$.
Since by Proposition \ref{varietyreal} and Proposition \ref{varietypadic}, $V(F^\prime)$ is in fact Zariski-dense in $V$ for all $F^\prime\in\CC_\mfp(F)$, $\mfpinS$,
we can assume without loss of generality that $V$ is affine.
Since $F/E$ is algebraic, $V$ is defined over a finite subextension $F_0$ of $F/E$. 
Let $W={\rm res}_{F_0/E}(V)$ be the Weil restriction of $V$
and let $F_1$ be the Galois closure of $F_0/E$.
Then $W$ is a variety defined over $E$
and there are $\sigma_1,\dots,\sigma_n\in\Gal(E)$ with $\sigma_1={\rm id}_{\tilde{E}}$
such that 
$W$ is isomorphic over $F_1$ to $\prod_{i=1}^n V^{\sigma_i}$,
and the projection onto the first factor
$W\rightarrow V^{\sigma_1}= V$ is defined over $F_0$,
cf.~\cite[10.6.2]{FJ3}.
Since $V$ is smooth, also $W$ is smooth.

Since $E\subseteq F\subseteq\KtotS$ and $E\subseteq F_1\subseteq\KtotS$,
Lemma~\ref{KtotSmaxtotS}(2) implies that $\CC_\mfp(E)=\CC_\mfp(F)=\CC_\mfp(F_1)$ for all $\mfpinS$.
In particular, if $E^\prime\in\CC_\mfp(E)$, then
$F_1\subseteq E^\prime$.
Let $E^\prime\in\CC_\mfp(E)$. 
Then $\sigma_i^{-1}(E^\prime)\in\CC_\mfp(E)=\CC_\mfp(F)$,
so $V(\sigma_i^{-1}(E^\prime))\neq\emptyset$ by assumption.
Thus $V^{\sigma_i}(E^\prime)\neq\emptyset$ for all $i$,
and therefore $W(E^\prime)\neq\emptyset$, as $F_1\subseteq E^\prime$.
Since $E$ is $\PSCC$, $W(E)\neq\emptyset$,
so in particular $W(F)\neq\emptyset$.
Hence, since $F_0\subseteq F$,
it follows that
$V(F)\neq\emptyset$, as claimed.
\end{proof}

\begin{Definition}
If $\bfsigma=(\sigma_1,\dots,\sigma_e)\in\Gal(K)^e$,
we denote by $\tilde{K}(\bfsigma)$ 
the fixed field of the group
$\left<\sigma_1,\dots,\sigma_e\right>\leq\Gal(K)$ in $\tilde{K}$,
and let $\KtotS(\bfsigma)=\KtotS\cap\tilde{K}(\bfsigma)$.
We say that a statement holds for {\bf almost all} $\bfsigma\in\Gal(K)^e$
if the set of those $\bfsigma\in\Gal(K)^e$ for which it holds
has measure $1$ with respect to the unique Haar probability measure on $\Gal(K)^e$.
\end{Definition}

\begin{Proposition}[Geyer-Jarden]\label{JardenRazon}\label{GeyerJarden}
Let $\mfS$ be a finite set of local primes of a countable Hilbertian field $K$
of characteristic zero,
and let $e\geq0$.
Then for almost all $\bfsigma\in\Gal(K)^e$,
the field $\KtotS(\bfsigma)$ is $\PSCC$.
\end{Proposition}

\begin{proof}
By \cite[Theorem A]{GeyerJarden},
for almost all $\bfsigma\in\Gal(K)^e$, 
the maximal Galois extension
$M$ of $K$ inside $\KtotS(\bfsigma)$
satisfies a local-global principle with respect to the sets
$\mcS_\mfp'(M)$ of orderings and {\em arbitrary} valuations on $M$ lying over $\mfpinS$.
However, since $M/K$ is totally $\mfS$-adic by Lemma~\ref{KtotSmaxtotS}(1), 
and all extensions of $\mfp$ to $M$ are conjugate since $M/K$ is Galois,
in fact $\mcS_\mfp'(M)=\mcS_\mfp(M)$.
In other words, $M$ is $\PSCC$.
Since $K\subseteq M\subseteq\KtotS(\bfsigma)\subseteq\KtotS$,
the claim follows 
from Lemma~\ref{PSCCextension}.
\end{proof}

\begin{Proposition}[Haran--Jarden--Pop]\label{GalQtotpSigma}
Let $\mfS$ be a finite set of local primes of a countable Hilbertian field $K$
of characteristic zero,
and let $e\geq 0$.
Then for almost all $\bfsigma\in\Gal(K)^e$,
$\GalS(\KtotS(\bfsigma))$ 
is isomorphic to the deficient reduct of the $e$-free
semi-constant group pile of $(\Gal(K_\mfp))_\mfpinS$
over perfect profinite spaces $(T_\mfp)_\mfpinS$.
\end{Proposition}

\begin{proof}
This is proven in \cite{HJPd}.
Indeed, by \cite[Prop.~12.3]{HJPd}, for almost all $\bfsigma\in\Gal(K)^e$,
the field $M=\KtotS(\bfsigma)$ satisfies condition (1) of \S10 of that work.
In the proof of \cite[Prop.~11.2]{HJPd} it is proven that
in this case 
${\bf Gal}(M,\mfS):=(G,\Gal(\tilde{K}(\bfsigma))^G,\mcG_\mfp)_\mfpinS$,
where $\GalS(M)=(G,\mcG_\mfp)_\mfpinS$,
is a so called 
\textquoteleft Cantor group pile over $(\Gal(K_\mfp))_\mfpinS$\textquoteright.
By \cite[Cor.~6.2]{HJPd} and \cite[Prop.~6.3]{HJPd}, every Cantor group pile over $(\Gal(K_\mfp))_\mfpinS$
is isomorphic to the group pile $\mathbf{G}_T$ of \cite[Prop.~5.3]{HJPd},
which is exactly the $e$-free semi-constant group pile of 
$(\Gal(K_\mfp))_\mfpinS$
over Cantor spaces $(T_\mfp)_\mfpinS$. 
Since Cantor spaces are perfect and the deficient reduct of ${\bf Gal}(M,\mfS)$ is $\GalS(M)$,
the claim follows.
\end{proof}

\section{Axiomatization}
\label{sec:axiomatization}

\noindent
We now present the axiomatization of the theory of almost all $\KtotS(\bfsigma)$.

\begin{Setting}
Let $\mfS$ be a finite set of local primes of a countable Hilbertian field $K$ of characteristic zero.
\end{Setting}

\begin{Definition}\label{DefZKR}
For a set $S\subseteq\tilde{K}$ let
$N_K(S) = \left\{ f\in K[X] \colon f\mbox{ has no root in }S \right\}$.
\end{Definition}

\begin{Definition}\label{DefTalg}
Let the $\mcLr(K)$-theory $T_{{\rm alg},\mfS}$ consist of the following sentences:
\begin{enumerate}
 \item For each $f\in N_K(\KtotS)$
    the sentence
$$
\neg(\exists x)(f(x)=0).
$$
 \item For each $\mfpinS$ and each $f\in N_{K}(R_\mfp(\KtotS))$
   the sentence
$$
 \neg(\exists x)(\varphi_{{\rm R},\mfp}(x)\wedge f(x)=0),
$$
where $\varphi_{{\rm R},\mfp}$ is the formula of Proposition~\ref{PSCC:holomorphy}.
\end{enumerate}
\end{Definition}

\begin{Lemma}\label{Kochen}
A $\PSCC$ field $F\supseteq K$ is a model of $T_{{\rm alg},\mfS}$
if and only if
$F\cap\tilde{K}\subseteq\KtotS$ and $F/F\cap\tilde{K}$ is totally $\mfS$-adic.
\end{Lemma}

\begin{proof}
Since $F$ is $\PSCC$, 
$\varphi_{{\rm R},\mfp}(F)=R_\mfp(F)$
for each $\mfpinS$
by Proposition~\ref{PSCC:holomorphy}.
Let $L=F\cap\tilde{K}$.
Since $L/K$ is algebraic,
$L$ is $\mfS$-SAP by Proposition~\ref{PSCC:SAP}.

Suppose that $F$ satisfies $T_{{\rm alg},\mfS}$.
By (1), $L\subseteq\KtotS$.
If $\mfpinS$,
then
$R_\mfp(L)=R_\mfp(\KtotS)\cap L$
by Lemma~\ref{QtotpKochen}.
Since $R_\mfp(\KtotS)$ is $\Gal(K)$-invariant, 
(2) implies that $R_\mfp(F)\cap\tilde{K}\subseteq R_\mfp(\KtotS)$.
Therefore, $R_\mfp(F)\cap L \subseteq R_\mfp(\KtotS)\cap L = R_\mfp(L)$,
so $F/L$ is totally $\mfS$-adic by Proposition~\ref{totallySadic}.

Conversely, suppose that
$L\subseteq\KtotS$ and $F/L$ is totally $\mfS$-adic.
Since $L\subseteq\KtotS$, $F$ satisfies (1).
By Proposition~\ref{totallySadic}, $R_\mfp(F)\cap L=R_\mfp(L)$.
So since $R_\mfp(F)\cap L=R_\mfp(F)\cap \tilde{K}$ and
$R_\mfp(L)=R_\mfp(\KtotS)\cap L$
by Lemma~\ref{QtotpKochen},
$F$ satisfies (2).
\end{proof}

\begin{Definition}\label{DefTheory}
Let the $\mcLr(K)$-theory
$T_{{\rm tot},\mfS,e}$
consist of the following axioms:
\begin{enumerate}
 \item[(0)] The axioms for fields and the positive diagram of $K$, cf.~\cite[7.3.1]{FJ3}.
 \item[(1)] The theory $T_{\PSCC}$ (Proposition \ref{thm:PSCC1}).
 \item[(2)] The theory $T_{{\rm C},\mfS,e}^{\rm ring}$ (Definition \ref{DefTCantorring}).
 \item[(3)] The theory $T_{{\rm alg},\mfS}$ (Definition \ref{DefTalg}).
\end{enumerate}
\end{Definition}

\begin{Lemma}\label{LemmaTheory}
A field $F\supseteq K$ is a model of $T_{{\rm tot},\mfS,e}$ 
if and only if it satisfies the following conditions:
\begin{enumerate}
\item $F$ is $\PSCC$.
\item $\GalS(F)$ is an $e$-free C-pile.
\item $F\cap\tilde{K}\subseteq\KtotS$ and $F/F\cap\tilde{K}$ is totally $\mfS$-adic.
\end{enumerate}
\end{Lemma}

\begin{proof}
See Proposition \ref{thm:PSCC1} for (1),
Proposition \ref{AxCantor} and Corollary \ref{interpretCDM} for (2),
and Lem\-ma~\ref{Kochen} for (3).
\end{proof}

\begin{Lemma}\label{mainlemmaSigma}
Let $K\subseteq L\subseteq E,F$ be fields such that
the following conditions are satisfied.
\begin{enumerate}
 \item $E$ and $F$ are models of $T_{{\rm tot},\mfS,e}$.
 \item $E/L$ and $F/L$ are regular and totally $\mfS$-adic.
 \item $E$ is countable and $F$ is $\aleph_1$-saturated.
 \item $L$ is $\mfS$-quasi-local.
\end{enumerate}
Then there exists an $L$-embedding $i\function E\rightarrow F$ with
$F/i(E)$ regular and totally $\mfS$-adic.
\end{Lemma}

\begin{proof}
By (1), 
$E$ and $F$ are $\PSCC$ (Lemma~\ref{LemmaTheory}(1)),
so in particular $\mfS$-quasi-local (Proposition \ref{PSCC:quasilocal}).
Let $\mbG=\GalS(F)$,
$\mbB=\GalS(E)$,
and
$\mbA=\GalS(L)$.
Also by (1), 
$\mathbf{G}$ and $\mathbf{B}$ are $e$-free C-piles (Lemma~\ref{LemmaTheory}(2)).
By (2) and (4),
the restriction maps ${\rm res}_{\tilde{F}/\tilde{L}}\function\mathbf{G}\rightarrow\mathbf{A}$ and
${\rm res}_{\tilde{E}/\tilde{L}}\function\mathbf{B}\rightarrow\mathbf{A}$
are rigid epimorphisms of group piles (Lemma~\ref{rigid}).
So $({\rm res}_{\tilde{F}/\tilde{L}},{\rm res}_{\tilde{E}/\tilde{L}})$
is a rigid $e$-generated deficient embedding problem for $\mathbf{G}$.

By (3), $E$ is countable, so $\mathbf{B}$ has countable rank,
and $F$ is $\aleph_1$-saturated,
so $\GalS(F)$ is $\aleph_1$-cosaturated by Corollary \ref{cosaturated}.
Hence, by Proposition~\ref{Cantorsaturated} 
there exists an epimorphism $\gamma\function\mathbf{G}\rightarrow\mathbf{B}$
such that ${\rm res}_{\tilde{E}/\tilde{L}}\circ\gamma={\rm res}_{\tilde{F}/\tilde{L}}$.
By Proposition~\ref{PSCEmbedding},
this gives an $L$-embedding
$i\function E\rightarrow F$ such that
$\gamma={\rm res}_{\tilde{F}/\widetilde{i(E)}}$.
Hence, since $\gamma$ is an epimorphism of group piles,
$F/i(E)$ is regular and totally $\mfS$-adic by Lemma~\ref{rigid}.
\end{proof}

The proof of the following proposition follows the proof of \cite[20.3.3]{FJ3}.

\begin{Proposition}\label{characterization}
Let $E,F\supseteq K$ be models of $T_{{\rm tot},\mfS,e}$
with $E\cap\tilde{K}\cong_K F\cap\tilde{K}$.
Then $E\equiv_K F$, i.e.\ $E$ and $F$ are elementarily equivalent in $\mathcal{L}_{\rm ring}(K)$.
\end{Proposition}

\begin{proof}
Assume without loss of generality that $L:=E\cap\tilde{K}=F\cap\tilde{K}$.
Let $E^*,F^*$ be $\aleph_1$-saturated elementary extensions of $E$ resp.~$F$. 
By Lemma~\ref{LemmaTheory},
the fields $E,F,E^*,F^*$ are $\PSCC$,
in particular also $\mfS$-SAP (Proposition \ref{PSCC:SAP}),
and the extensions $E/L$, $F/L$, $E^*/L$, $F^*/L$ are regular and totally $\mfS$-adic.

By L\"owenheim-Skolem, there exists a countable elementary subfield
$E_0$ of $E^*$ that contains $L$.
Then also $E_0/L$ is regular and totally $\mfS$-adic,
and $E_0$ is a model of $T_{{\rm tot},\mfS,e}$, in particular $\PSCC$ and $\mfS$-quasi-local (Proposition \ref{PSCC:quasilocal}).
Since $L/K$ is algebraic, $L$ is $\mfS$-quasi-local (Proposition~\ref{PSCC:quasilocal}).
Therefore, by Lemma~\ref{mainlemmaSigma}, there exists an $L$-embedding
$\alpha_0\function E_0\rightarrow F^*$ with $F^*/\alpha_0(E_0)$ regular and totally $\mfS$-adic.

Identify $E_0$ with $\alpha_0(E_0)$.
Let $F_0$ be a countable elementary subfield of $F^*$ that contains $E_0$.
Then $F_0/E_0$ is regular and totally $\mfS$-adic, 
and $F_0$ is a model of $T_{{\rm tot},\mfS,e}$.
Also $E^*/E_0$ is regular and totally $\mfS$-adic by Proposition~\ref{PSCC:elementaryext}.
Hence, by Lemma~\ref{mainlemmaSigma}, there is an $E_0$-embedding
$\beta_0\function F_0\rightarrow E^*$ with $E^*/\beta_0(F_0)$ regular and totally $\mfS$-adic.

Now iterate this process
to construct a tower of countable fields 
$E_0\subseteq F_0\subseteq E_1\subseteq F_1\subseteq\dots$ 
such that each $E_i$ is an elementary subfield of $E^*$
and each $F_i$ is an elementary subfield of $F^*$.
Then $M:=\bigcup_{i\in\mathbb{N}} E_i=\bigcup_{i\in\mathbb{N}} F_i$ is an elementary subfield of both $E^*$ and $F^*$,
see for example \cite[7.4.1(b)]{FJ3},
so $E^*\equiv_M F^*$. 
In particular, $E^*\equiv_K F^*$,
hence $E\equiv_K F$.
\end{proof}

\begin{Lemma}\label{algebraicpart}
If $F$ is a model of
$T_{{\rm tot},\mfS,e}$, then
$L:=F\cap\tilde{K}\subseteq\KtotS$ and
$\rk(\Gal(\KtotS/L))\leq e$.
\end{Lemma}

\begin{proof}
Let $\mbG=\GalS(F)$, and $\mbA=\GalS(L)$.
By Lemma~\ref{LemmaTheory}(3),
$L\subseteq\KtotS$ and $F/L$ is totally $\mfS$-adic.
Since $L/K$ is algebraic,
$L$ is $\mfS$-quasi-local (Proposition \ref{PSCC:quasilocal}),
so the restriction $\mathbf{G}\rightarrow\mathbf{A}$ is an epimorphism of group piles
by Lemma~\ref{rigid}.
By Lemma~\ref{LemmaTheory}(2),
$\mathbf{G}$ is an $e$-free C-pile.
In particular, it is $e$-generated.
Thus, by Lemma~\ref{quotientrank}, also $\mathbf{A}$ is $e$-generated.
By Lemma~\ref{KtotSmaxtotS}(2),
$\CC_\mfp(L)=\CC_\mfp(\KtotS)$, for all $\mfpinS$.
Thus, $\mbA^\prime=\GalS(\KtotS)^\prime$.
But $\GalS(\KtotS)$ is self-generated
by the definition of $\KtotS$,
so $\mathbf{A}^\prime=\GalS(\KtotS)$.
Therefore, ${\rm Gal}(\KtotS/L)=A/A^\prime=\bar{A}$ is generated by $e$ elements.
\end{proof}

\begin{Definition}
Let $T_{{\rm almost},\mfS,e}$ denote the set of all
$\mcLr(K)$-sentences that are true
in almost all fields $\KtotS(\bfsigma)$, $\bfsigma\in\Gal(K)^e$.
\end{Definition}

The proof of the following result follows 
the proof of \cite[20.5.4]{FJ3}.

\begin{Theorem}\label{almostaxioms}
The theory $T_{{\rm tot},\mfS,e}$ is an axiomatization of $T_{{\rm almost},\mfS,e}$,
i.e.~these two theories have the same models.
\end{Theorem}

\begin{proof}
First note that every model of 
$T_{{\rm almost},\mfS,e}$
is a field containing $K$.
By Definition~\ref{DefTheory}(0),
the same holds for every model of
$T_{{\rm tot},\mfS,e}$.
Next observe that
almost all $\KtotS(\bfsigma)$ satisfy $T_{{\rm tot},\mfS,e}$, as Lemma \ref{LemmaTheory} shows:
For almost all $\bfsigma$, $\KtotS(\bfsigma)$
satisfies (1) by Proposition~\ref{GeyerJarden},
(2) by Proposition~\ref{GalQtotpSigma} and Proposition~\ref{constantCantor},
and (3) trivially.
Thus, every model of $T_{{\rm almost},\mfS,e}$ is a model of $T_{{\rm tot},\mfS,e}$.

Conversely, let $E$ be a model of $T_{{\rm tot},\mfS,e}$
and let $L=E\cap\tilde{K}$.
If we can construct a model $F$ of $T_{{\rm almost},\mfS,e}$
with $F\cap\tilde{K}\cong_K L$, then
$E\equiv_K F$ by Proposition~\ref{characterization},
so $E$ is a model of $T_{{\rm almost},\mfS,e}$ and we are done.
Lemma~\ref{algebraicpart} implies that $L\subseteq\KtotS$ and 
there exist $\tau_1,\dots,\tau_e\in{\rm Gal}(\KtotS/K)$
that generate ${\rm Gal}(\KtotS/L)$.
Let $\mathcal{N}$ be the set of finite Galois extensions of $K$ inside $\KtotS$.
For each $N\in\mathcal{N}$,
the set 
\begin{eqnarray*}
 \Sigma(N)&:=&\{\bfsigma\in\Gal(K)^e \colon {\rm res}_{\tilde{K}/N}(\sigma_i)={\rm res}_{\tilde{K}/N}(\tau_i),\;i=1,\dots,e\}\\
    &\subseteq&\{\bfsigma\in\Gal(K)^e \colon \KtotS(\bfsigma)\cap N\cong_K L\cap N\}
\end{eqnarray*}
has positive Haar measure.
If $N_1,\dots,N_r\in\mathcal{N}$,
then $N_1\cdots N_r\in\mathcal{N}$ and
$\Sigma(N_1)\cap\dots\cap\Sigma(N_r)=\Sigma(N_1\cdots N_r)$.
Hence, by \cite[7.6.1]{FJ3}, there exists
an ultrafilter $\mathcal{D}$ on $\Gal(K)^e$ which
contains each of the sets $\Sigma(N)$, $N\in\mathcal{N}$,
and all sets of measure~$1$.
Let 
$$
 F=\prod_{\smallbfsigma\in\Gal(K)^e}\KtotS(\bfsigma)/\mathcal{D}
$$ 
be the ultraproduct, and let $M=F\cap\tilde{K}$.
Since $\mathcal{D}$ contains all sets of measure $1$,
and almost all $\KtotS(\bfsigma)$ are models of $T_{{\rm almost},\mfS,e}$,
$F$ is a model of $T_{{\rm almost},\mfS,e}$ by \L os' theorem.
Furthermore, $M\subseteq\KtotS$,
and $M\cap N\cong L\cap N$ for each $N\in\mathcal{N}$,
since $\mathcal{D}$ contains $\Sigma(N)$.
Therefore, $M\cong_K L$, see for example \cite[20.6.3]{FJ3}, as claimed.
\end{proof}

\section{Decidability}

\noindent
We now use the axiomatization of the previous section to prove the decidability of $T_{{\rm almost},\mfS,e}$.
The method follows closely the proof of Jarden-Kiehne in \cite[Chapter 20.6]{FJ3} 
for the special case $\mfS=\emptyset$.

\begin{Definition}
A set $X\subseteq\mathbb{N}^n$ is {\bf recursive} 
if the characteristic function of $X$ is a recursive function, cf.~\cite[Ch.~8.5]{FJ3}.
If $\mcL$ is a countable language with a fixed embedding $\mcL\rightarrow\mathbb{N}$, then
an $\mcL$-theory $T$ 
is {\bf decidable} (or recursive)
if the set $T$,
identified with a subset of $\mathbb{N}$ via a G\"odel numbering, is recursive,
\cite[Ch.~8.6]{FJ3}.

A {\bf presented field} is a countable field $K$ together with an
injection $\rho\function K\rightarrow\mathbb{N}$ such that the
images of the graphs of addition and multiplication
are recursive. 
If $\mcL$ is a finite language containing $\mcLr$, 
then the injection $\rho\function K\rightarrow\mathbb{N}$
induces an injection of the set of $\mcL(K)$-formulas into $\mathbb{N}$.
We refer to this injection when we call an $\mcL(K)$-theory decidable. 

If $K$ is a presented field, one can
inject the ring of polynomials $K[X]$ into $\mathbb{N}$ via a recursive pairing function
$\mathbb{N}\times\mathbb{N}\rightarrow\mathbb{N}$.
We say that $K$ has a {\bf splitting algorithm}
if the set of irreducible polynomials in $K[X]$ is a recursive subset of $K[X]$.
In that case, one can recursively factor elements of $K[X]$ into irreducible factors.
\end{Definition}

\begin{Definition}
A prime $\mfp$ of a presented field $\rho\function K\rightarrow\mathbb{N}$
is {\bf recursive}
if the set $\rho(\mcO_\mfp)\subseteq\mathbb{N}$ is recursive.
\end{Definition}

\begin{Setting}\label{set:splitting}
From now on,
let $\mfS$ be a finite set of recursive local primes of a presented countable Hilbertian field $K$
of characteristic zero that has a splitting algorithm.
\end{Setting}

\begin{Example}\label{Ex:numberfield}
Every finite set of primes $\mfS$ of a number field $K$ satisfies Setting \ref{set:splitting}:
Choose any standard representation $K=\Q^n\hookrightarrow\mathbb{N}^{2n}\hookrightarrow\mathbb{N}$ via recursive pairing functions.
Then $K$ is countable, Hilbertian \cite[13.4.2]{FJ3}, and has a splitting algorithm \cite[19.1.3(b), 19.2.4]{FJ3},
and every $\mfpinS$ is local and recursive.
This last fact is well-known, but for lack of reference we sketch a proof:
Since $\mathcal{O}_\mfp$ is existentially definable in $K$ 
(see for example \cite[p.~212]{Rumely} for the archimedean
and \cite[4.2.4, 4.3.4]{Shlapentokh} for the $p$-adic case), it is recursively enumerable.
Since $K\setminus\mcO_\mfp=-\mcO_\mfp\setminus\{0\}$ in the archimedean
and $K\setminus\mcO_\mfp=(\pi\mcO_\mfp)^{-1}$ in the $p$-adic case, where $\pi$ is a uniformizer at $\mfp$,
also $K\setminus\mcO_\mfp$ is recursively enumerable,
and so $\mcO_\mfp$ is recursive.
\end{Example}

\begin{Lemma}\label{splitting}
The sets $N_K(\KtotS)$ and $N_K(R_\mfp(\KtotS))$, $\mfpinS$, are recursive.
\end{Lemma}

\begin{proof}
Let $f=\sum_{i=0}^na_iX^i\in K[X]$ be given. We have to decide whether $f$ has a root in $\KtotS$ resp.~in $R_\mfp(\KtotS)$.
Using the splitting algorithm we can assume without loss of generality that $f$ is irreducible.
Since $\KtotS=\bigcap_{\mfq\in\mfS}\bigcap\CC_\mfq(K)$ and $R_\mfp(\KtotS)=\KtotS\cap\bigcap_{K'\in\CC_\mfp(K)}R_\mfp(K')$ by Lemma \ref{QtotpKochen}(3), and all elements of $\CC_\mfp(K)$ are $K$-conjugate, 
it suffices to decide whether $f$ has all roots in $K_\mfq$ (resp.~in $R_\mfq(K_\mfq)$) for all $\mfq\in\mfS$.

By \cite[Prop.~8.2]{AFPSCC} we can compute a universal $\mcL_{{\rm ring},\mfq}$-formula $\varphi_1$ such
that $f$ has all roots in $K_\mfq$ (resp.~in $R_\mfq(K_\mfq)$) iff $(K,\mcO_\mfq)\models\varphi_1(\mathbf{a})$.
By applying the same result to the negation, we get a universal $\mcL_{{\rm ring},\mfq}$-formula $\varphi_2$
such that $f$ has all roots in $K_\mfq$ (resp.~in $R_\mfq(K_\mfq)$) iff  $(K,\mcO_\mfq)\not\models\varphi_2(\mathbf{a})$.
Therefore, since $\mathcal{O}_\mfq$ is recursive, the set of such $f$ is both r.e.~and co-r.e., hence recursive.
\end{proof}

\begin{Lemma}\label{Ttotperecursive}
The theory $T_{{\rm tot},\mfS,e}$ is recursive.
\end{Lemma}

\begin{proof}
We check that the different sets of axioms in Definition \ref{DefTheory} are recursive:

(0): Since $K$ is a presented field, the positive diagram of $K$ is recursive.

(1): In order to check that $T_\PSCC$ is recursive, we have to look into the definition \cite[Def.~9.1]{AFPSCC}.
Part (1) of these axioms is recursive, because \cite[Def.~6.5]{AFPSCC} (1), (3), (4), (5) consist of finitely many sentences,
and (2) is recursive because $\mcO_\mfp$ is.
Part (2) of \cite[Def.~9.1]{AFPSCC} is recursive because the map $\eta_n\mapsto(\hat{\eta}_n)_{\mfp,\forall}$ 
is recursive by \cite[Prop.~8.4]{AFPSCC}.

(2): The theory $T_{{\rm C},\mfS,e}^{\rm ring}$ is recursive
since $T_{{\rm C},\mfS,e}^{\rm co}$ is obviously recursive and the map 
$\varphi\mapsto\varphi_{\rm ring}$ of Corollary \ref{interpretCDM} is recursive.

(3): Since $N_K(\KtotS)$ and $N_K(R_\mfp(\KtotS))$ are recursive by Lemma~\ref{splitting},
the theory $T_{{\rm alg},\mfS}$ is recursive.
\end{proof}

\begin{Definition}\label{Deftest}
The set of {\bf test sentences}
is the smallest set of $\mcLr(K)$-sentences that
contains all sentences of the form
$(\exists x)(f(x)=0)$,
where $f\in K[X]$ is a polynomial that completely decomposes over $\KtotS$,
and is closed under negations, conjunctions, and disjunctions.
\end{Definition}

\begin{Lemma}\label{testsentences}
Let $E,F\supseteq K$ be models of $T_{{\rm tot},\mfS,e}$.
Then $E\equiv_K F$ if and only if $E$ and $F$ satisfy the same test sentences.
\end{Lemma}

\begin{proof}
Trivially, if $E$ and $F$ are elementarily equivalent over $K$, then they satisfy the same test sentences.
Conversely, assume that $E$ and $F$ satisfy the same test sentences, 
and let $E_0=E\cap\tilde{K}$ and $F_0=F\cap\tilde{K}$.
By Lemma~\ref{LemmaTheory}(3), $E_0\subseteq\KtotS$ and $F_0\subseteq\KtotS$.
Let $f\in K[X]$ be an irreducible polynomial.
If $f$ does not completely decompose over $\KtotS$, then it has no root in $\KtotS$,
so it has no root in $E_0$ and it has no root in $F_0$.
If $f$ completely decomposes over $\KtotS$, then $(\exists x)(f(x)=0)$ is a test sentence.
Hence, $f$ has a root in $E_0$ if and only if it has a root in $F_0$.
Therefore, $E_0\cong_K F_0$, see for example \cite[20.6.3]{FJ3}.
By Proposition~\ref{characterization}, $E\equiv_K F$ .
\end{proof}

\begin{Lemma}\label{testsentencesrecursive}
The set of test sentences is recursive.
\end{Lemma}

\begin{proof}
Given a polynomial $f\in K[X]$,
one can decide whether $(\exists x)(f(x)=0)$ is a test sentence because $N_K(\KtotS)$
is recursive by Lemma \ref{splitting}.
Induction on the structure of formulas then shows that
the set of test sentences is recursive.
\end{proof}

\begin{Definition}
For each $\mcLr(K)$-sentence $\theta$ let
$\Sigma_{\mfS,e}(\theta)=\{\bfsigma\in\Gal(K)^e \colon \KtotS(\bfsigma)\models\theta\}$.
We denote by $\mu$ the unique Haar probability measure on $\Gal(K)^e$.
\end{Definition}

\begin{Lemma}\label{testrecursive}
Let $\lambda$ be a test sentence.
Then $\Sigma_{\mfS,e}(\lambda)$ is open-closed in $\Gal(K)^e$ and
$\mu(\Sigma_{\mfS,e}(\lambda))\in\mathbb{Q}$.
The map $\lambda\mapsto\mu(\Sigma_{\mfS,e}(\lambda))$ 
from test sentences to $\Q$ is recursive.
\end{Lemma}

\begin{proof}
Let $f_1,\dots,f_n\in K[X]$ be the polynomials
occurring in $\lambda$.
Their splitting field $L_\lambda$ is a finite Galois extension of $K$ inside $\KtotS$.
Let $L/K$ be a Galois extension with $L_\lambda\subseteq L\subseteq\KtotS$ (this is needed for the induction).
Then
$\KtotS(\bfsigma)\cap L=L({\rm res}_{\tilde{K}/L}(\bfsigma))$
for each $\bfsigma\in\Gal(K)^e$.
Let 
$$
 \Sigma_{L,\lambda}=\left\{ \mbox{\boldmath$\tau$}\in\Gal(L/K)^e \colon L(\mbox{\boldmath$\tau$})\models\lambda \right\}.
$$
We claim that
$\Sigma_{\mfS,e}(\lambda)=\{\bfsigma\in\Gal(K)^e \colon {\rm res}_{\tilde{K}/L}(\bfsigma)\in\Sigma_{L,\lambda}\}$.
Indeed, if $\lambda$ is of the form $(\exists x)(f(x)=0)$, where 
$f\in K[X]$ completely decomposes over $\KtotS$,
then 
$$
 \Sigma_{L,\lambda}=\{\mbox{\boldmath$\tau$}\in\Gal(L/K)^e \colon f\mbox{ has a zero in }L(\mbox{\boldmath$\tau$})\}.
$$
Since $L$ contains all roots of $f$, 
$\KtotS(\bfsigma)\models\lambda$ if and only if
$\KtotS(\bfsigma)\cap L\models\lambda$,
so the claim is true in that case.
Induction on the structure of $\lambda$ shows that the claim
holds for all test sentences $\lambda$.
Thus, $\Sigma_{\mfS,e}(\lambda)$ is open-closed, in particular measurable. Furthermore,
$$
 \mu(\Sigma_{\mfS,e}(\lambda))=\frac{|\Sigma_{L_\lambda,\lambda}|}{[L_\lambda:K]^e}\in\Q
$$ 
is computable since $K$ has a splitting algorithm, see for example \cite[19.3.2]{FJ3}.
\end{proof}

\begin{Theorem}\label{TheoremSigma}
Under Setting \ref{set:splitting}, the following holds:
\begin{enumerate}
 \item For every $\mcLr(K)$-sentence $\theta$,
         $\Sigma_{\mfS,e}(\theta)$ is $\mu$-measurable, and
        $\mu(\Sigma_{\mfS,e}(\theta))\in\mathbb{Q}$.
 \item The map $\theta\mapsto\mu(\Sigma_{\mfS,e}(\theta))$ from $\mcLr(K)$-sentences to $\Q$
        is recursive.
\end{enumerate}
In particular, the theory
$T_{{\rm almost},\mfS,e}$ of almost all fields $\KtotS(\bfsigma)$, $\bfsigma\in\Gal(K)^e$,
is decidable.
\end{Theorem}

\begin{proof}
By Theorem~\ref{almostaxioms},
$T_{{\rm tot},\mfS,e}\models T_{{\rm almost},\mfS,e}$
and
$T_{{\rm almost},\mfS,e}\models T_{{\rm tot},\mfS,e}$.
By Lemma~\ref{testsentences}
and \cite[7.8.2]{FJ3},
for every $\mcLr(K)$-sentence $\theta$ there exists
a test sentence $\lambda$ such that
the sentence $\theta\leftrightarrow\lambda$ is in $T_{{\rm almost},\mfS,e}$.
In particular, $\Sigma_{\mfS,e}(\theta)$ and $\Sigma_{\mfS,e}(\lambda)$ differ only by a zero set.
Lemma~\ref{testrecursive} implies that $\Sigma_{\mfS,e}(\lambda)$ is $\mu$-measurable and $\mu(\Sigma_{\mfS,e}(\lambda))\in\Q$,
so also $\Sigma_{\mfS,e}(\theta)$ is $\mu$-measurable and $\mu(\Sigma_{\mfS,e}(\theta))=\mu(\Sigma_{\mfS,e}(\lambda))\in\mathbb{Q}$.

Since $T_{{\rm tot},\mfS,e}\models T_{{\rm almost},\mfS,e}$,
we have
$T_{{\rm tot},\mfS,e}\models\theta\leftrightarrow\lambda$.
The set of test sentences is recursive by Lemma~\ref{testsentencesrecursive}.
By Lemma~\ref{Ttotperecursive}, the theory
$T_{{\rm tot},\mfS,e}$ is recursive,
so the set of consequences of $T_{{\rm tot},\mfS,e}$ is recursively enumerable.
Therefore, there is a recursive map $\theta\mapsto\lambda_{\theta}$
from $\mcLr(K)$-sentences to test sentences such that
$\theta\leftrightarrow\lambda_\theta$ is in $T_{{\rm almost},\mfS,e}$
for every $\theta$.
In particular, $\mu(\Sigma_{\mfS,e}(\theta))=\mu(\Sigma_{\mfS,e}(\lambda_{\theta}))$.
Since also the map $\lambda\mapsto\mu(\Sigma_{\mfS,e}(\lambda))$ from test sentences to $\Q$
is recursive by Lemma~\ref{testrecursive},
so is the composition
$\theta\mapsto\lambda_\theta\mapsto\mu(\Sigma_{\mfS,e}(\lambda_\theta))=\mu(\Sigma_{\mfS,e}(\theta))$.
\end{proof}

\begin{Remark}
By applying this theorem to Example \ref{Ex:numberfield}
we finally deduce Theorem~\ref{TheoremIntro1} from the introduction.
\end{Remark}

\begin{Remark}
Note that the assumption that the primes in $\mfS$ are recursive
is {\em necessary}.
Indeed,$R_\mfp(\KtotS)$ is $K$-definable in $\KtotS$
for each $\mfpinS$, cf.~Proposition \ref{PSCC:holomorphy}. An element $x\in K$ lies in $\mcO_\mfp$
if and only if $x\in R_\mfp(\KtotS)$,
so the decidability of the complete $\mcLr(K)$-theory of $\KtotS$ implies
that $\mcO_\mfp$ is recursive.
On the other hand we do not know whether the assumption that $K$ has 
a splitting algorithm is necessary.

The theorem does certainly not hold anymore if we allow
$\mfS$ to be an arbitrary (possibly infinite) set of recursive local primes of $K$.
In fact, although there exist trivial examples of Hilbertian fields $K$
with an infinite set of local primes $\mfS$ such that $\KtotS$ is decidable,
we do not know any infinite set of primes $\mfS$ of $K=\Q$ for which the theorem holds.
Moreover, \cite[Example 10.4]{JardenPop} gives an example of an infinite
set of primes $\mfS$ of $\mathbb{Q}$ {\em that has Dirichlet density zero},
but $\Q^\mfS=\Q$, hence
$T_{{\rm almost},\mfS,e}={\rm Th}(\Q)$
is undecidable.
\end{Remark}

\begin{Remark}
With the machinery developed here and some additional work one can 
show that Theorem \ref{TheoremSigma} holds with $\KtotS(\bfsigma)$ 
replaced by the maximal Galois extension $\KtotS[\bfsigma]$ of $K$ inside $\KtotS(\bfsigma)$.
We refer the interested reader to \cite[Chapter 5]{AFDiss},
where this is shown for number fields $K$.
\end{Remark}

\appendix
\section{Real closed fields}
\label{app:real}

\noindent
We recall the notion of real closed fields and
quote some well known results from \cite{Prestel1}.

Let $K$ be a field.
A {\bf positive cone} of $K$ is a semiring $P\subseteq K$
such that $P\cup(-P)=K$ and $P\cap(-P)=\{0\}$.
An {\bf ordering} of $K$ is a total order $\leq$ on $K$
such that $\{x\in K \colon x\geq 0\}$ is a positive cone.
The map that assigns to an ordering the corresponding positive cone
induces a natural bijection between the orderings of $K$ and the positive cones of $K$.
An {\bf ordered field} is a field $K$ together with an ordering.
An ordering $\leq$ of $K$ is {\bf archimedean} if for every $x\in K$ there exists $y\in\mathbb{N}\subseteq K$ with $x< y$.
A {\bf pre-positive cone} of $K$ is 
a semiring $P\subseteq K$ such that
$K^2\subseteq P$ and $-1\notin P$.

\begin{Lemma}\label{prepositive}
Each pre-positive cone of $K$ is the intersection of the
positive cones of $K$ containing it.
In particular, 
each pre-positive cone of $K$ is contained in a positive cone of $K$.
\end{Lemma}

\begin{proof}
See \cite[1.6]{Prestel1}.
\end{proof}

A field is {\bf real closed} if it has an ordering but
each proper algebraic extension has no ordering.
A real closed field $K$ has a unique ordering, 
given by the positive cone $K^2$, \cite[3.2]{Prestel1}.
A real closed field $F$ is a {\bf real closure} of an ordered field $K$
if $F$ is an algebraic extension of $K$ and the unique ordering of $F$ extends 
the ordering of $K$.
Any ordered field $K$ has a real closure,
which is unique up to $K$-isomorphism, \cite[3.10]{Prestel1}.
If $L$ is a finite extension of an ordered field $K$, then
the extensions of the ordering of $K$ to $L$ bijectively correspond to
the $K$-embeddings of $L$ into a fixed real closure of $K$, \cite[3.12]{Prestel1}.

\begin{Lemma}\label{algclosedreal}
A field which is algebraically closed in a real closed field 
is real closed.
\end{Lemma}

\begin{proof}
See \cite[3.13]{Prestel1}.
\end{proof}

\begin{Proposition}[Artin-Schreier]\label{ArtinSchreier}
A field $K$ is real closed if and only if $\Gal(K)\cong\mathbb{Z}/2\mathbb{Z}$.
\end{Proposition}

\begin{proof}
This follows from \cite[VI.9.3]{L-A} and \cite[3.3]{Prestel1}.
\end{proof}

\begin{Proposition}\label{varietyreal}
Let $V$ be a smooth absolutely irreducible variety over a real closed field $K$.
If $V(K)\neq\emptyset$, then $V(K)$ is Zariski-dense in $V$.
\end{Proposition}

\begin{proof}
This follows for example from \cite[7.10]{Prestel1}.
\end{proof}

\begin{Proposition}[Tarski]\label{QERCF}
The $\mcL$-theory of real closed ordered fields is model complete. 
\end{Proposition}

\begin{proof}
See \cite[3.3.15, 3.3.16]{Marker}.
\end{proof}

\section{$p$-adically closed fields}
\label{app:padic}
  
\noindent
We recall the notion of $p$-adically closed fields
and quote some well known results from \cite{PR}
and some properties of the absolute Galois group
of a $p$-adically closed field.

A valuation $v$ on a field $K$ of characteristic zero
with residue field of characteristic $p>0$
and corresponding valuation ring $\mcO$
is a {\bf $p$-valuation of $p$-rank $d\in\mathbb{N}$}
if 
${\rm dim}_{\mathbb{F}_p}\mcO/p\mcO=d$.
We also say that the valued field $(K,v)$ is a {\bf $p$-valued} field.

The residue field $\bar{K}_v$ of a $p$-valued field $(K,v)$ is finite,
and the value group $v(K^\times)$ is discrete and $v(p)\in\mathbb{Z}$.
If $e=v(p)$ and $f=[\bar{K}_v:\mathbb{F}_p]$, then $d=ef$, \cite[p.~15]{PR}.
We call $(p,e,f)$ the {\bf type} of $(K,v)$.
Thus, if two $p$-valued fields have the same type, then they have the same $p$-rank.
If $L/K$ is an extension of $p$-valued fields,
then $L$ and $K$ have the same $p$-rank if and only if
they have the same type.
In that case, this type is also the type of each intermediate extension of $L/K$.

A $p$-valued field is {\bf $p$-adically closed} if it has no proper $p$-valued algebraic extension
of the same $p$-rank.
Every $p$-adically closed valued field $(K,v)$ has a unique $p$-valuation, \cite[6.15]{PR}.
We therefore also call $K$ $p$-adically closed.
A {\bf $p$-adic closure} of a $p$-valued field $(K,v)$ is an
algebraic extension of $(K,v)$ which is $p$-adically closed of the same $p$-rank as $(K,v)$.
A $p$-valued field $(K,v)$ is $p$-adically closed if and only if
it is Henselian and the value group $v(K^\times)$ is a $\mathbb{Z}$-group, \cite[3.1]{PR}.
Here, an ordered abelian group $\Gamma$ is a {\bf $\mathbb{Z}$-group}
if it is discrete and $(\Gamma:n\Gamma)=n$ for each $n\in\mathbb{N}$.
Any $p$-valued field $(K,v)$ has a $p$-adic closure.
A $p$-adic closure of $(K,v)$ is unique up to $K$-isomorphism 
if and only if $v(K^\times)$ is a $\mathbb{Z}$-group, \cite[3.2]{PR}.

\begin{Lemma}\label{algclosedpadic}
If a field is algebraically closed in a $p$-adically closed field $K$,
then it is $p$-adically closed of the same $p$-rank as $K$.
\end{Lemma}

\begin{proof}
See \cite[3.4]{PR}.
\end{proof}

\begin{Proposition}\label{varietypadic}
Let $V$ be a smooth absolutely irreducible variety over a $p$-adically closed field $K$.
If $V(K)\neq\emptyset$, then $V(K)$ is Zariski-dense in $V$.
\end{Proposition}

\begin{proof}
This follows for example from \cite[7.8]{PR}.
\end{proof}

\begin{Proposition}\label{QEpCF}
The $\mcL$-theory of $p$-adically closed valued fields of $p$-rank $d$ is model complete.
\end{Proposition}

\begin{proof}
See \cite[5.1, 5.2, 5.6, 5.4]{PR}.
\end{proof}

A {\bf $p$-adic field} is the completion of a $p$-valued number field,
i.e.~a finite extension of $\hat{\Q}_p$.
A $p$-adic field $F$ is $p$-adically closed of $p$-rank $[F:\hat{\Q}_p]$,
\cite[p.~21]{PR}.

\begin{Lemma}\label{elempadic}
Every $p$-adically closed field is elementarily equivalent to a $p$-adic field.
\end{Lemma}

\begin{proof}
Let $E$ be $p$-adically closed, and let $K=\tilde{\Q}\cap E$, $K_0=\tilde{\Q}\cap\hat{\Q}_p$.
By Lemma~\ref{algclosedpadic}, $K$ is $p$-adically closed of the same $p$-rank as $E$,
so $K_0=\tilde{\Q}\cap\hat{\Q}_p\subseteq K$.
Then $[K:K_0]<\infty$, c.f.~\cite[2.9]{PR}, so $F:=K\hat{\Q}_p$ is a $p$-adic field.
Since $K$ is algebraically closed in $F$, $K$ and $F$ have the same $p$-rank by Lemma~\ref{algclosedpadic}.
Therefore, $E\equiv K\equiv F$ by model completeness (Proposition~\ref{QEpCF}).
\end{proof}

\begin{Proposition}\label{Galpadic}
Let $K$ be $p$-adically closed. 
Then $\Gal(K)$ is fi\-nite\-ly generated and torsion-free.
\end{Proposition}

\begin{proof}
First, if $K$ is a $p$-adic field, then $\Gal(K)$ is
finitely generated 
and ${\rm cd}_l(\Gal(K))=2$ for every $l$, \cite[7.4.1, 7.1.8(i)]{NSW},
so in particular it is torsion-free.
Since every $p$-adically closed field $K$ is elementarily equivalent
to a $p$-adic field $K_0$ (Lemma~\ref{elempadic}), 
and $\Gal(K_0)$ is finitely generated,
$\Gal(K)\cong\Gal(K_0)$ has all the asserted properties, cf.~\cite[20.4.6]{FJ3}.
\end{proof}

\begin{Proposition}[Neukirch-Pop-Efrat-Koenigsmann]\label{NPEK}
Let $K$ be $p$-adically closed,
and let $L$ be a field.
If ${\rm Gal}(K)\cong{\rm Gal}(L)$,
then $L$ is $p$-adically closed of the same type as $K$.
\end{Proposition}

\begin{proof}
By Lemma~\ref{elempadic},
$K$ is elementarily equivalent to a $p$-adic field $K_0$,
and $\Gal(K)\cong\Gal(K_0)$ by Proposition~\ref{Galpadic} and \cite[20.4.6]{FJ3}.
By \cite[Theorem 4.1]{Koenigsmannpadic},
if $\Gal(K_0)\cong\Gal(L)$,
then $L$ is $p$-adically closed.
But $\Gal(L)$ determines the type of $L$,
see for example \cite[Lemma~1]{JardenRitter}.
\end{proof}

\section{Model theory of absolute Galois groups}
\label{app:absGal}

We now review how to translate statements about the inverse system $S(\Gal(F))$ of the absolute Galois group of a field into statements about the field itself.
Such a translation is claimed already in the unpublished \cite{CherlinDriesMacintyre2}, but without proof.
In fact, the careless reader might be tempted to assume that one can
prove \cite[Lemma 17]{CherlinDriesMacintyre2}
by giving an honest interpretation of the $\omega$-sorted structure $S(\Gal(F))^\omega$
in $F$, but experts think that this is actually impossible.
On the other hand, the abstract proof presented in \cite{Chatzidakisforking} is 
complete, but does not allow to conclude that the translation is recursive.
For this section, we fix $\mfS=\emptyset$ and let $\mathcal{L}_{\rm co}=\mathcal{L}_{{\rm co},\emptyset}$.

\begin{Definition}
A explained in \cite[\S1]{CherlinDriesMacintyre2} one can encode finite extensions of $F$ of degree $n$ by $n^3$-tuples $\mathbf{a}\in F^{n^3}$. 
We denote the extension defined by $\mathbf{a}\in F^{n^3}$ by $F_\mathbf{a}$. 
If $\mathbf{a}_1\in F^{n_1^3}$ and $\mathbf{a}_2\in F^{n_2^3}$ encode extensions of degree $n_1$ resp.\ $n_2$, then
an $F$-embedding of $F_{\mathbf{a}_1}$ into $F_{\mathbf{a}_2}$ is encoded by an $n_1n_2$-tuple $\mathbf{b}\in F^{n_1n_2}$.
\end{Definition}

\begin{Definition}
An {\bf admissible} sequence of length $\lambda\leq\omega$ and degrees $(n_i)_{i<\lambda}$ is a sequence 
$\mathbf{a}=\mathbf{a}_1\mathbf{a}_2\dots$ of elements of $F$ such that for $i<\lambda$, the tuple $\mathbf{a}_i$ encodes
\begin{enumerate}
\item a finite Galois extension $F_{\mathbf{a},i}$ of $F$ of degree at most $n_i$, 
\item an automorphism $\sigma_{\mathbf{a},i}\in{\rm Gal}(F_{\mathbf{a},i}/F)$,
\item a finite extension $F_{\mathbf{a},i}^*$ of $F$,
\item an embedding $\epsilon_{\mathbf{a},i}:F_{\mathbf{a},i}\rightarrow F_{\mathbf{a},i}^*$,
\item and for each $j=0,\dots,i$ an embedding $\epsilon_{\mathbf{a},j,i}:F_{\mathbf{a},j}^*\rightarrow F_{\mathbf{a},i}^*$,
\end{enumerate}
such that $F_{\mathbf{a},i}^*$ is the compositum of
all $\epsilon_{\mathbf{a},j,i}(\epsilon_{\mathbf{a},j}(F_{\mathbf{a},j}))$, $j\leq i$,
and
for $k\leq j\leq i$, $\epsilon_{\mathbf{a},i,i}={\rm id}_{F_{\mathbf{a},i}^*}$ and $\epsilon_{\mathbf{a},j,i}\circ\epsilon_{\mathbf{a},k,j}=\epsilon_{\mathbf{a},k,i}$.

We intentionally do not keep track of the size of the tuples in order to avoid an overload of notation.
We will also handle finite and infinite sequences rather sloppily.
\end{Definition}

\begin{Lemma}
For each $\lambda<\omega$ and $\mathbf{n}=(n_i)_{i<\lambda}$ there is an $\mathcal{L}_{\rm ring}$-formula
$\alpha_{\lambda,\mathbf{n}}$ such that for a tuple $\mathbf{a}$ of elements of $F$,
$F\models\alpha_{\lambda,\mathbf{n}}(\mathbf{a})$ iff $\mathbf{a}$ is admissible of length $\lambda$ and degrees $\mathbf{n}$.
\end{Lemma}

\begin{proof}
This is clear, see also \cite[\S1]{CherlinDriesMacintyre2}.
\end{proof}

\begin{Definition}
Every $x\in S(\Gal(F))$ is of the form $gN$ for some open normal subgroup $N\lhd\Gal(F)$
and we denote by $F_x$ the finite extension of $F$ that is the fixed field of $N$,
and by $\sigma_x\in{\rm Gal}(F_x/F)$ the automorphism corresponding to $x$ under the natural isomorphism
$\Gal(F)/N\cong\Gal(F_x/F)$ induced by restriction.
\end{Definition}

\begin{Definition}\label{def:compatible}
A sequence $\mathbf{x}=(x_i)_{i<\lambda}$ of elements of $S(\Gal(F))$ and an admissible sequence $\mathbf{a}$ of length $\lambda$ 
are {\bf compatible} if there exist 
$F$-isomorphisms $\phi_i:F_{x_0}\cdots F_{x_i}\rightarrow F_{\mathbf{a},i}^*$
with $\phi_i(F_{x_i})=\epsilon_{\mathbf{a},i}(F_{\mathbf{a},i})$  -- 
so $\phi_i$ induces an isomorphism $\tilde{\phi}_i:F_{x_i}\rightarrow F_{\mathbf{a},i}$ --
and $\tilde{\phi}_i\circ\sigma_{x_i}= \sigma_{\mathbf{a},i}\circ\tilde{\phi}_i$ for all $i<\lambda$,
such that for $j\leq i$, $\phi_i|_{F_{x_0}\cdots F_{x_j}}= \epsilon_{\mathbf{a},j,i}\circ\phi_j$.
\end{Definition}

\begin{Lemma}\label{lem:compatible}
Let $\mathbf{x}$ and $\mathbf{a}$ be compatible sequences of length $\lambda<\omega$.
\begin{enumerate}
\item[(a)] For every $x_{\lambda}\in S(\Gal(F))$ there exists $\mathbf{a}'$ such that
$\mathbf{a}\mathbf{a}'$ is admissible of length $\lambda+1$ and $(x_0,\dots,x_\lambda)$ and $\mathbf{a}\mathbf{a}'$ are compatible.
\item[(b)] For every $\mathbf{a}'$ such that $\mathbf{a}\mathbf{a}'$ is admissible of length $\lambda+1$ there exists
 $x_{\lambda}\in S(\Gal(F))$ such that $(x_0,\dots,x_{\lambda})$ and $\mathbf{a}\mathbf{a}'$ are compatible.
\end{enumerate}
\end{Lemma}

\begin{proof}
 Let $\phi_i:F_{x_0}\cdots F_{x_i}\rightarrow F_{\mathbf{a},i}^*$ for $i<\lambda$ be as in Definition \ref{def:compatible}.
 
(a) 
Choose $\mathbf{a}'$ such that there are isomorphisms $\phi':F_{x_{\lambda}}\rightarrow F_{\mathbf{a}\mathbf{a}',\lambda}$
and $\phi_\lambda:F_{x_0}\cdots F_{x_{\lambda}}\rightarrow F_{\mathbf{a}\mathbf{a}',\lambda}^*$,
and $\epsilon_{\mathbf{a}\mathbf{a}',\lambda}=\phi_\lambda\circ\phi'^{-1}$,
$\epsilon_{\mathbf{a}\mathbf{a}',i,\lambda}=\phi_\lambda\circ\phi_i^{-1}$,
$\sigma_{\mathbf{a},\lambda}=\phi'\circ\sigma_{x_\lambda}\circ\phi'^{-1}$.

(b) Extend the isomorphism $\phi_{\lambda-1}^{-1}: F_{\mathbf{a},\lambda-1}^*\rightarrow F_{x_0}\cdots F_{x_{\lambda-1}}$ to an embedding
$\psi:F_{\mathbf{a}\mathbf{a}',\lambda}^*\rightarrow(F_{x_0}\cdots F_{x_\lambda})^{\tilde{}}$ with $\psi\circ\epsilon_{\mathbf{a}\mathbf{a}',\lambda-1,\lambda}=\phi_{\lambda-1}^{-1}$,
and choose $x_{\lambda}$ such that $F_{x_{\lambda}}=\psi(\epsilon_{\mathbf{a}\mathbf{a}',\lambda}(F_{\mathbf{a}\mathbf{a}',\lambda}))$
and $\sigma_{x_\lambda}$ corresponds to $\sigma_{\mathbf{a}\mathbf{a}',\lambda}$ under $\psi\circ\epsilon_{\mathbf{a}\mathbf{a}',\lambda}$.
\end{proof}

\begin{Definition}\label{def:coring}
We assign to each bounded $\mathcal{L}_{\rm co}$-formula $\varphi(v_0,\dots,v_{\lambda-1})$ and tuple
$\mathbf{n}=(n_i)_{i<\omega}$ an $\mathcal{L}_{\rm ring}$-formula $\varphi_{{\rm ring},\mathbf{n}}(\mathbf{u})$,
where $\mathbf{u}=\mathbf{u}_0\mathbf{u}_1\dots$, as follows:
\begin{enumerate}
\item If $\varphi$ is $v_i\leq v_j$, then $\varphi_{{\rm ring},\mathbf{n}}$ expresses
that $\alpha_{k+1,\mathbf{n}}(\mathbf{u})$, where $k=\max\{i,j\}$, and that $\epsilon_{\mathbf{u},i,k}(\epsilon_{\mathbf{u},i}(F_{\mathbf{u},i}))\supseteq\epsilon_{\mathbf{u},j,k}(\epsilon_{\mathbf{u},j}(F_{\mathbf{u},j}))$,
so that $\epsilon_{\mathbf{u}}$ induces an embedding $F_{\mathbf{u},j}\rightarrow F_{\mathbf{u},i}$.
\item If $\varphi$ is $v_i\sqsubseteq v_j$,
then $\varphi_{{\rm ring},\mathbf{n}}$ expresses
that 
$\alpha_{k+1,\mathbf{n}}(\mathbf{u})$, where $k=\max\{i,j\}$, and that
$\epsilon_\mathbf{u}$ induces an embedding $F_{\mathbf{u},j}\rightarrow F_{\mathbf{u},i}$,
and, under this embedding, $\sigma_{\mathbf{u},i}|_{F_{\mathbf{u},j}}=\sigma_{\mathbf{u},j}$.
\item If $\varphi$ is $P(v_{i_1},v_{i_2},v_{i_3})$, then $\varphi_{{\rm ring},\mathbf{n}}$ expresses that
$\alpha_{k+1,\mathbf{n}}(\mathbf{u})$, where $k=\max\{i_1,i_2,i_3\}$, and that
the three embeddings $\epsilon_{\mathbf{u},i_\nu,k}\circ\epsilon_{\mathbf{u},i_\nu}$,
$\nu=1,2,3$, induce isomorphisms between the $F_{\mathbf{u},i_\nu}$, and,
under these isomorphisms, $\sigma_{\mathbf{u},i_1}\circ\sigma_{\mathbf{u},i_2}=\sigma_{\mathbf{u},i_3}$.
\item If $\varphi$ is $v_i=v_j$, then $\varphi_{{\rm ring},\mathbf{n}}$ expresses that
$\alpha_{k+1,\mathbf{n}}(\mathbf{u})$, where $k=\max\{i,j\}$, and that
 $\epsilon_\mathbf{u}$ induces an isomorphism $\phi:F_{\mathbf{u},i}\rightarrow F_{\mathbf{u},j}$,
 and $\phi\circ\sigma_{\mathbf{u},i}=\sigma_{\mathbf{u},j}\circ\phi$.
\item If $\varphi$ is $v_i\in G_n$, then $\varphi_{{\rm ring},\mathbf{n}}$ 
expresses that 
$\alpha_{i+1,\mathbf{n}}(\mathbf{u})$ and that
$[F_{\mathbf{u},i}:F]\leq n$.
\item If $\varphi$ is of the form $\psi\wedge\eta$, then $\varphi_{{\rm ring},\mathbf{n}}$ is $\psi_{{\rm ring},\mathbf{n}}\wedge\eta_{{\rm ring},\mathbf{n}}$,
and if $\varphi$ is of the form $\neg\psi$, then $\varphi_{{\rm ring},\mathbf{n}}$ is $\neg\psi_{{\rm ring},\mathbf{n}}$.
\item If $\varphi$ is of the form $(\exists v_i\in G_n)(\psi)$ 
and $\lambda\geq i$ is minimal such that the free variables in $\psi$ are among $v_0,\dots,v_{\lambda}$,
then, after renumbering the variables we can assume that $i=\lambda$, and
 $\varphi_{{\rm ring},\mathbf{n}}(\mathbf{u})$ 
 is 
$$
 (\exists\mathbf{u}')(\alpha_{\lambda+1,\mathbf{n}'}(\mathbf{u}\mathbf{u}')\wedge \psi_{{\rm ring},\mathbf{n}'}(\mathbf{u}_0\dots\mathbf{u}_{\lambda-1}\mathbf{u}')),
$$ 
 where $\mathbf{n}':=(n_0,\dots,n_{\lambda-1},n)$.
\end{enumerate}
If $\Sigma$ is a ranked set of coformulas in the variables $(v_i)_{i<\lambda}$, we let 
$\mathbf{n}:=\mathbf{n}_\Sigma:=(n_i)_{i<\lambda}$ with $n_i$ minimal such that 
the formula $G_{n_i}(v_i)$ is in $\Sigma$, and let
$\Sigma_{\rm ring}$ consist of all $\alpha_{i,\mathbf{n}}$, for $i<\lambda$, and all
$\varphi_{{\rm ring},\mathbf{n}}$ where $\varphi\in\Sigma$.
\end{Definition}

\begin{Lemma}\label{lem:maincoring}
Let $\mathbf{x}$, $\mathbf{a}$ be compatible sequences of length $\lambda\leq\omega$ and degrees $\mathbf{n}$ and let
$\varphi(v_0,\dots,v_{\lambda-1})$ be a bounded $\mathcal{L}_{\rm co}$-formula. Then
$S(\Gal(F))\models\varphi(\mathbf{x})$ if and only if $F\models\varphi_{{\rm ring},\mathbf{n}}(\mathbf{a})$.
\end{Lemma}

\begin{proof}
This is proven by case distinction according to Definition \ref{def:coring}:
In case (1)-(5), the claim follows immediately from the compatibility.
In case (6), the claim is trivial by induction.
In case (7), the claim follows by induction and Lemma \ref{lem:compatible}.
\end{proof}

\begin{Proposition}\label{prop:coring}
Let $\Sigma$ be a ranked set of coformulas in the variables $v_0,v_1,\dots$. 
Then $\Sigma$ is cosatisfied in $\Gal(F)$ if and only if $\Sigma_{\rm ring}$ is satisfied in $F$.
\end{Proposition}

\begin{proof}
Let $\mathbf{n}=\mathbf{n}_\Sigma$ as in Definition \ref{def:coring}.

If $\Sigma$ is cosatisfied in $\Gal(F)$ by $\mathbf{x}=(x_0,x_1,\dots)$, then $[F_{x_i}:F]\leq n_i$ for all $i$.
By applying Lemma \ref{lem:compatible} iteratively we
get an admissible sequence $\mathbf{a}$ of degrees $\mathbf{n}$ such that $\mathbf{x}$ and $\mathbf{a}$ are compatible.
In particular, $F\models\alpha_{i,\mathbf{n}}(\mathbf{a})$ for every $i$.
For every $\varphi\in\Sigma$, $S(\Gal(F))\models\varphi(\mathbf{x})$ implies $F\models\varphi_{{\rm ring},\mathbf{n}}(\mathbf{a})$
by Lemma \ref{lem:maincoring}, so $\Sigma_{\rm ring}$ is satisfied in $F$.

Conversely, if $\Sigma_{\rm ring}$ is satisfied by a sequence $\mathbf{a}$, then,
since $\Sigma_{\rm ring}$ contains all $\alpha_{i,\mathbf{n}}$, $\mathbf{a}$ is admissible of degrees $\mathbf{n}$.
Lemma \ref{lem:compatible} gives a sequence $\mathbf{x}$ in $S(\Gal(F))$ such that $\mathbf{x}$ and $\mathbf{a}$ are compatible.
For every $\varphi\in\Sigma$, $F\models\varphi_{{\rm ring},\mathbf{n}}(\mathbf{a})$ implies that $S(\Gal(F))\models\varphi(\mathbf{x})$
by Lemma \ref{lem:maincoring}, so $\Sigma$ is satisfied in $F$.
\end{proof}

\begin{Corollary}
There is a recursive map $\varphi\mapsto\varphi_{\rm ring}$ from bounded $\mathcal{L}_{\rm co}$-sentences to 
$\mathcal{L}_{\rm ring}$-sentences such that for any field $F$,
$S(\Gal(F))\models\varphi$ if and only if $F\models\varphi_{\rm ring}$.
\end{Corollary}

\begin{proof}
This is the special case $\Sigma=\{\varphi\}$ of Proposition \ref{prop:coring}.
\end{proof}

\begin{Corollary}\label{cor:sat0}
If $F$ is $\aleph_1$-saturated, then $\Gal(F)$ is $\aleph_1$-cosaturated.
\end{Corollary}

\begin{proof}
Let $\Sigma$ be a countable ranked set of coformulas in the variables $v_0,v_1,\dots$ with parameters in $S(\Gal(F))$
for which every finite subset is cosatisfied in $\Gal(F)$,
and let $\mathbf{n}=\mathbf{n}_\Sigma$.
For a finite subset 
$$
 \Sigma_0'=\{(\varphi_1)_{{\rm ring},\mathbf{n}},\dots,(\varphi_m)_{{\rm ring},\mathbf{n}},\alpha_{i_1,\mathbf{n}},\dots,\alpha_{i_l,\mathbf{n}}\}\subseteq\Sigma_{\rm ring}
$$ 
choose $\lambda\geq\max\{i_1,\dots,i_l\}$ such that the free variables of the $\varphi_i$ are among $v_0,\dots,v_\lambda$, 
and let 
$$
 \Sigma_0:=\{\varphi_1,\dots,\varphi_m,G_{n_0}(v_0),\dots,G_{n_\lambda}(v_\lambda)\}.
$$
Since $(\varphi_i)_{{\rm ring},\mathbf{n}}=(\varphi_i)_{{\rm ring},\mathbf{n}_{\Sigma_0}}$ we see that $\Sigma_0'\subseteq(\Sigma_0)_{\rm ring}$,
so since $\Sigma_0$ is cosatisfied in $\Gal(F)$, Proposition \ref{prop:coring} gives that $(\Sigma_0)_{\rm ring}$, and hence $\Sigma_0'$ is satisfied in $F$.
Therefore, since $F$ is $\aleph_1$-saturated, $\Sigma_{\rm ring}$ is satisfied in $F$,
hence $\Sigma$ is cosatisfied in $\Gal(F)$, again by Proposition~\ref{prop:coring}.
\end{proof}

\section*{Acknowledgements}

\noindent
The author would like to thank Zo\'e Chatzidakis and Franziska Jahnke for helpful discussions concerning the model theory of profinite groups.
This work builds on Chapters 3 and 4 of the author's Ph.D. thesis \cite{AFDiss} supervised by Moshe Jarden at Tel Aviv University.
As such, it was supported by the European Commission under contract MRTN-CT-2006-035495.

\end{document}